\documentclass[11pt,letterpaper]{amsart}

\usepackage[margin=1.15in]{geometry}
\usepackage[T1]{fontenc}
\usepackage{lmodern}
\usepackage{microtype}

\usepackage{amsmath,amssymb,amsthm,mathtools}
\numberwithin{equation}{section}
\allowdisplaybreaks

\usepackage{enumitem}
\usepackage{booktabs}
\usepackage{array}
\usepackage{longtable}
\mathtoolsset{showonlyrefs=true}

\usepackage[hidelinks]{hyperref}

\hypersetup{
  pdftitle={Critical Gaussian Multiplicative Chaos on the Circle Is Rajchman},
  pdfauthor={Yin Cai, Bonan Chen, Xiang Fang, Feng Guo},
  pdfsubject={The Rajchman property of critical Gaussian multiplicative chaos on the circle},
  pdfkeywords={critical Gaussian multiplicative chaos, Rajchman measure, Fourier coefficients, derivative martingale, Seneta--Heyde normalization, conditional Bernstein inequality},
  pdfstartview=FitH,
  pdfdisplaydoctitle=true
}

\setlist[enumerate,1]{
  label=\textup{(\roman*)},
  leftmargin=2.5em
}
\setlist[enumerate,2]{
  label=\textup{(\alph*)},
  leftmargin=2.5em
}
\setlist[itemize]{
  leftmargin=2.0em
}

\newtheorem{theorem}{Theorem}[section]
\newtheorem{proposition}[theorem]{Proposition}
\newtheorem{lemma}[theorem]{Lemma}
\newtheorem{corollary}[theorem]{Corollary}

\theoremstyle{definition}
\newtheorem{definition}[theorem]{Definition}

\theoremstyle{remark}

\newcommand{\proofheading}[1]{
  \par\vspace{0.3\baselineskip}
  \noindent\textbf{#1}\enspace
}

\title{Critical Gaussian Multiplicative Chaos on the Circle Is Rajchman}

\author[Y. Cai]{Yin Cai}
\address{School of Mathematics\\
Hangzhou Normal University\\
Hangzhou 311121\\
P. R. China}
\email{cymath@hznu.edu.cn}

\author[B. Chen]{Bonan Chen}
\address{School of Mathematical Sciences\\ 
Soochow University\\ 
Suzhou 215006\\ 
P. R. China}
\email{bnchen@suda.edu.cn}

\author[X. Fang]{Xiang Fang}
\address{Department of Applied Mathematics\\
National Yang Ming Chiao Tung University\\
Hsinchu 30010\\
Taiwan}
\email{xfang@nycu.edu.tw}

\author[F. Guo]{Feng Guo}
\address{School of Mathematics \\ 
Nanjing University of Aeronautics and Astronautics\\ 
Nanjing 210016\\ 
P. R. China}  
\email{70207994@nuaa.edu.cn}



\date{}

\begin{document}

\begin{abstract}
We prove that the Fourier coefficients of the canonical critical Gaussian
multiplicative chaos on the circle vanish almost surely at infinity. More
precisely, let \(M_\phi^{\mathrm{crit}}\) be the canonical critical chaos
associated with the centered circle field \(\phi\) of covariance
\begin{equation*}
\mathbb{E}[\phi(\theta)\phi(\theta')]
=
\log\frac{1}{\lvert e^{i\theta}-e^{i\theta'}\rvert}.
\end{equation*}
Then, almost surely,
\begin{equation*}
\widehat{M_\phi^{\mathrm{crit}}}(n)\longrightarrow0
\qquad
\text{as }\lvert n\rvert\to\infty.
\end{equation*}
This resolves the almost-sure critical Rajchman problem for the canonical
circle field. Since critical chaos has Fourier dimension zero almost surely,
no positive polynomial Fourier-decay rate can hold; the theorem therefore
exhibits qualitative Fourier cancellation beyond the regime of positive
Fourier dimension.

The proof addresses two coupled difficulties: the heavy, nonuniform cell masses of critical chaos and the need to control exponentially many frequencies in each dyadic annulus. 
For an auxiliary periodized compact-range star-scale field, a derivative-rooted Bessel regression yields weighted small-cell summability and moving-tail control of exceptional large cells. 
After conditioning at a coarse scale below the Fourier scale, finite-range independence and conditional Bernstein concentration reduce uniform control of the terminal Fourier coefficients over each dyadic annulus to a spatial-variation estimate for a coarse predictable measure.
A smooth positive-definite covariance correction and critical-chaos uniqueness then transfer the Rajchman property to the canonical critical chaos of the exact circle field.
\end{abstract}
\maketitle

\section{Introduction and main result}
\label{sec:introduction}

\subsection{The critical Fourier problem and the main theorem}
\label{subsec:critical-rajchman-problem}

Critical Gaussian multiplicative chaos is a canonical random singular measure at the boundary of nondegeneracy.
In one dimension it is carried by sets of Hausdorff dimension zero \cite{Powell2021}, and hence has Fourier dimension zero. 
Thus, no positive polynomial Fourier-decay exponent can hold.
This leaves a qualitative pathwise question: do the individual Fourier coefficients of the critical measure nevertheless vanish at infinity?
We answer this question affirmatively for the canonical critical chaos associated with the exact log-correlated field on the circle.

Gaussian multiplicative chaos, introduced by Kahane \cite{Kahane1985}, constructs random measures from logarithmically correlated Gaussian fields;
see \cite{Powell2021,RhodesVargas2014} for background.
In one dimension, the usual renormalized exponentials converge to a nontrivial measure for parameters \(0<\gamma<\sqrt{2}\).
At the critical value \(\gamma=\sqrt{2}\), this normalization converges to zero,
and a nontrivial measure is recovered either through derivative normalization or through the Seneta--Heyde factor \(\sqrt{\log(1/\varepsilon)}\);
see \cite{DRSVDerivative,DRSVRenormalization,Powell2021}.

We work with the standard exact circle model.
Write
\[
\mathbb{T}
=
\mathbb{R}/(2\pi\mathbb{Z}),
\]
and denote its geodesic distance by \(d_{\mathbb{T}}\). 
Let \(\phi\) be the centered generalized Gaussian field on \(\mathbb{T}\) with covariance
\begin{equation}
\label{eq:exact-circle-covariance}
\mathbb{E}[\phi(\theta)\phi(\theta')]
=
\log\frac{1}{\lvert e^{i\theta}-e^{i\theta'}\rvert}.
\end{equation}
Let \(\rho\in C_c^\infty(\mathbb{R})\) be nonnegative with integral one, write \(\rho_\varepsilon\) for its periodized rescaling,
and set \(\phi_\varepsilon=\rho_\varepsilon*\phi\).

\begin{definition}
\label{def:critical-circle-chaos}
The canonical critical Gaussian multiplicative chaos associated with \eqref{eq:exact-circle-covariance} is the random finite measure \(M_\phi^{\mathrm{crit}}\) characterized by
\begin{equation}
\label{eq:critical-circle-chaos-normalization}
\sqrt{\log(1/\varepsilon)}
\exp\left(
\sqrt{2}\phi_\varepsilon(\theta)
-
\mathbb{E}[\phi_\varepsilon(\theta)^2]
\right)\mathrm{d}\theta
\longrightarrow
M_\phi^{\mathrm{crit}}(\mathrm{d}\theta)
\end{equation}
in probability in the weak topology as \(\varepsilon\downarrow0\).
The limit is independent of the admissible mollifier.
\end{definition}

The existence, nontriviality, mollifier independence, and atomlessness in Definition~\ref{def:critical-circle-chaos} follow from the general critical chaos theory;
see \cite[Theorem~5.3]{JunnilaSaksmanWebb2019} for the first three properties and \cite[Section~2.6]{Powell2021} for atomlessness.
The normalization in \eqref{eq:critical-circle-chaos-normalization} is fixed throughout the paper.

For a finite Borel measure \(\mu\) on \(\mathbb{T}\), define
\[
\widehat{\mu}(n)
=
\int_{\mathbb{T}}e^{in\theta}\,\mu(\mathrm{d}\theta),
\qquad
n\in\mathbb{Z}.
\]

\begin{definition}
\label{def:rajchman-measure}
A finite Borel measure \(\mu\) on \(\mathbb{T}\) is called \emph{Rajchman} if
\[
\widehat{\mu}(n)
\longrightarrow0
\qquad
\text{as }\lvert n\rvert\to\infty.
\]
\end{definition}

For absolutely continuous measures, the Rajchman property follows from the Riemann--Lebesgue lemma.
For singular measures it is an additional cancellation property:
atomlessness alone gives only an averaged Fourier conclusion and does not imply pointwise decay of the coefficients.
In particular, the fact that critical chaos is carried by sets of
Hausdorff dimension zero does not decide the Rajchman question.

Our main result is the following.

\begin{theorem}[Critical circle GMC is Rajchman]
\label{thm:main-circle-rajchman}
Let \(M_\phi^{\mathrm{crit}}\) be the canonical critical chaos associated with \eqref{eq:exact-circle-covariance}.
Then, almost surely,
\begin{equation}
\label{eq:main-circle-rajchman}
\lim_{\lvert n\rvert\to\infty}
\widehat{M_\phi^{\mathrm{crit}}}(n)
=
0.
\end{equation}
\end{theorem}

\subsection{Previous work}
\label{subsec:main-theorem-previous-work}

For \(0\leq\gamma<\sqrt{2}\), let \(M_\gamma\) denote the canonical subcritical chaos associated with the exact circle field.
Garban and Vargas initiated the systematic study of the Fourier coefficients of one-dimensional Gaussian multiplicative chaos.
They proved that \(M_\gamma\) is almost surely Rajchman throughout the subcritical range; see \cite[Theorem~1.1]{GarbanVargas2026}.
Their proof is based on finite convolutions.
If \(d\geq2\) and
\[
\gamma<\sqrt{\frac{2(d-1)}{d}},
\]
then the \(d\)-fold convolution of \(M_\gamma\) is almost surely an \(L^1\)-function.
Its Fourier coefficients are \(\widehat{M_\gamma}(n)^d\), so the Riemann--Lebesgue lemma gives the Rajchman property.
For every fixed subcritical \(\gamma\), one can choose a finite \(d\), but the required convolution order diverges as \(\gamma\uparrow\sqrt{2}\).
Garban and Vargas therefore left the critical Rajchman property as an open problem and observed that their method would formally require an ill-defined infinite convolution;
see \cite[Section~1.2.2]{GarbanVargas2026}.

Arguin and Hamdan subsequently established a quantitative critical result for a different model.
Let \(\mu\) be the critical chaos associated with a compact-range star-scale invariant field on the unit interval,
and let \(c_n\) denote its \(n\)-th Fourier coefficient.
They proved that, for every \(\alpha<1/4\),
\begin{equation}
\label{eq:Arguin-Hamdan-comparison}
(\log n)^\alpha c_n
\longrightarrow 0
\qquad
\text{in probability as \(n\to\infty\)};
\end{equation}
see \cite[Theorem~1.1]{ArguinHamdan2025}.
In particular, \(c_n\to0\) in probability; more quantitatively, for every \(0<\alpha<1/4\), \(\lvert c_n\rvert=o_{\mathbb{P}}((\log n)^{-\alpha})\).

Theorem~\ref{thm:main-circle-rajchman} and \eqref{eq:Arguin-Hamdan-comparison} address different levels of the critical problem.
Arguin and Hamdan work with a star-scale invariant field on an interval,
whereas our theorem concerns the canonical critical chaos of the exact circle covariance \eqref{eq:exact-circle-covariance}.
Their result gives polylogarithmic decay in probability of every order strictly below \(1/4\),
whereas ours is qualitative but holds almost surely along the full positive and negative integer sequence.
Neither result follows formally from the other:
their convergence in probability does not provide a single probability-one event on which the full sequence converges, nor does it identify the exact circle chaos,
while qualitative almost-sure convergence does not imply any of the quantitative bounds in \eqref{eq:Arguin-Hamdan-comparison}.

Theorem~\ref{thm:main-circle-rajchman} therefore resolves the almost-sure critical Rajchman problem for the canonical circle field posed in \cite{GarbanVargas2026}.
To the best of our knowledge, no almost-sure full-sequence Rajchman theorem for critical real Gaussian multiplicative chaos was previously available.

\subsection{Why criticality requires a different mechanism}
\label{subsec:critical-difficulties}

First, the finite-convolution argument used in the subcritical regime has no critical endpoint.
For a fixed subcritical parameter, sufficiently many convolutions regularize the chaos measure until the Riemann--Lebesgue lemma applies,
but no finite convolution order is available at \(\gamma=\sqrt{2}\).
Theorem~\ref{thm:main-circle-rajchman} is therefore not obtained by passing to the endpoint in the subcritical proof.
The replacement for convolutional smoothing is a derivative-rooted nonconcentration estimate:
under the critical rooted law, the mass of successively smaller cells around a typical mass point satisfies a summability property strong enough to control the exceptional large cells that obstruct direct Fourier estimates.

Second, an almost-sure full-sequence statement requires substantially more than control of a single Fourier coefficient.
At the critical parameter, cell masses are heavy and highly nonuniform, so a small number of atypical cells can dominate estimates at the Fourier scale.
At the same time, a dyadic frequency annulus contains exponentially many integers.
Pointwise convergence in probability does not provide a failure probability that can be summed over all frequencies and scales.
We address this by combining rooted nonconcentration with conditioning at one coarse scale below the annulus.
Finite-range independence of the remaining increments and conditional Bernstein concentration then give a comparison with a coarse predictable measure that is uniform over the whole annulus.

Third, the finite-range structure needed for the annular argument is available in a periodized compact-range star-scale field,
whereas Theorem~\ref{thm:main-circle-rajchman} concerns the canonical critical chaos associated with the exact covariance \eqref{eq:exact-circle-covariance}.
Matching the limiting covariance distributions is not sufficient:
equality in law of the underlying generalized Gaussian fields does not by itself identify their critical chaos measures.
We resolve this by adding an independent smooth Gaussian field whose covariance supplies the missing positive-definite correction,
and then comparing the scale and mollifier approximations at the covariance level.
Critical-chaos uniqueness identifies the resulting measures on the required coupling and transfers the Rajchman property to the canonical circle chaos.

\subsection{Main probabilistic ideas}
\label{subsec:proof-mechanism}

The proof is first carried out for a periodized compact-range star-scale field \((X_t)_{t\geq0}\).
This auxiliary model has two features that are used separately.
On every sufficiently short arc, its complete space--time law agrees with the standard compact-range star-scale field on the line,
so the local critical-chaos construction and its rooted description apply without approximation.
On the other hand, the field added after time \(t\) has dependence range \(e^{-t}\),
which provides the conditional independence needed to control an entire Fourier annulus at once.
We write \(Z_t^\beta\) for the positive truncated critical martingale and \(Z^\beta\) for its terminal measure.

\proofheading{Derivative-rooted nonconcentration.}
The structure of this step is motivated by the spine-based endpoint mechanism developed for the Rajchman problem for Mandelbrot cascades \cite{CaiChengFangLiQuXiao2026Rajchman}:
sampling according to mass reduces the control of exceptional large cells to the behavior of a distinguished one-dimensional path.
In the present Gaussian setting, this principle is implemented through the derivative-rooted barrier process.

Under the probability measure obtained by rooting with \(Z^\beta\), the root is uniform on the circle and the process
\[
\beta+\sqrt{2}t-X_t(\theta)
\]
is a three-dimensional Bessel process.
An exact Gaussian regression of the field at a second point against the complete rooted path separates a uniformly controlled Gaussian remainder from a negative displacement generated by the Bessel spine.
Integrating the resulting conditional density estimate gives the principal nonconcentration statement:
if \(Q_m(\theta)\) is the dyadic cell of level \(m\) containing the root, then, for every fixed \(p\geq0\), almost surely,
\begin{equation}
\label{eq:introduction-rooted-summability}
\sum_{m=1}^{\infty}
m^p Z^\beta\bigl(Q_m(\theta)\bigr)
<\infty
\qquad
\text{for }Z^\beta\text{-almost every }\theta.
\end{equation}
This weighted small-cell summability is stronger than atomlessness: it controls the rate at which mass may accumulate around a typical mass point.
Theorem~\ref{thm:rooted-small-cell-summability} and Proposition~\ref{prop:moving-cell-tails} then show that,
for partitions at a level asymptotic to the Fourier scale,
both the realized mass of cells larger than \(b/N\) and its conditional expectation with respect to the coarse field vanish.
These moving-tail estimates remove the exceptional large cells from the subsequent annular concentration argument.

\proofheading{One-scale annular concentration.}
Let \(\Lambda_N=\{n\in\mathbb{Z}:2^N\leq\lvert n\rvert<2^{N+1}\}\), and condition at a coarse time \(t_N=r_N\log 2\), where \(r_N=N-O(\log N)\).
The cells at level \(r_N\) split into two alternating classes such that, within either class,
their terminal restrictions are conditionally independent given the coarse field.
After capping each cell at mass \(b/N\) and subtracting its conditional mean,
conditional Bernstein concentration gives a failure probability that remains summable after a union bound over all frequencies in \(\Lambda_N\).
The discarded part is negligible by the realized and predictable moving-tail estimates.
The resulting comparison is
\begin{equation}
\label{eq:introduction-annular-approximation}
\max_{n\in\Lambda_N}
\left|
\widehat{Z^\beta}(n)-\widehat{Z^\beta_{t_N}}(n)
\right|
\longrightarrow0
\qquad
\text{almost surely.}
\end{equation}
The measure-valued locality theorem and the adapted localization used in Section~\ref{sec:terminal-locality-annular} ensure that the conditional factorization
and concentration are applied with the correct sigma-algebras; they are safeguards for this step rather than additional sources of cancellation.

\proofheading{Decay of the coarse predictable measure.}
On the global barrier event, the killing indicator in \(Z_t^\beta\) disappears and the coarse predictable measure has a smooth density.
Stationarity and evenness of the covariance imply that the field value and its spatial derivative at a point are independent Gaussian variables.
After the exponential change of measure, this gives a direct bound on the expected total variation of the density.
A Borel--Cantelli argument at integer scale times, followed by periodic integration by parts,
therefore controls the Fourier coefficients of \(Z^\beta_{t_N}\) uniformly on \(\Lambda_N\).
The logarithmic separation between \(r_N\) and \(N\) leaves a polynomial margin;
Proposition~\ref{prop:predictable-annular-decay} gives an \(N^{-3}\) bound for this coarse measure for all sufficiently large \(N\).
Together with \eqref{eq:introduction-annular-approximation},
this proves the Rajchman property of the periodized derivative limit in Theorem~\ref{thm:periodized-derivative-rajchman}.
Proposition~\ref{prop:periodized-seneta-heyde} then identifies the positive critical scale measure as the deterministic multiple \(\sqrt{2/\pi}\,M'_\ast\).

\proofheading{Transfer to the canonical circle field.}
It remains to pass from the auxiliary scale decomposition to the exact circle covariance \eqref{eq:exact-circle-covariance}.
The Fourier coefficients of the limiting star-scale covariance can be computed explicitly.
Their deficit from the exact circle covariance is nonnegative and rapidly decreasing, hence is the covariance spectrum of an independent smooth stationary Gaussian field \(Y\);
see Lemma~\ref{lem:smooth-covariance-correction}.
If \(X_\ast\) is the limiting star-scale field and \(G_0\) is an independent Gaussian constant,
Proposition~\ref{prop:corrected-field-law} identifies \(X_\ast+Y\) in law with \(\phi+G_0\).
Multiplication by the smooth positive density generated by \(Y\) preserves the Rajchman property by Lemma~\ref{lem:rajchman-multiplier}.

Equality in law of the limiting fields does not by itself identify their critical chaos measures.
We therefore compare the scale approximations \(X_t+Y\) with the convolution regularizations of \(X_\ast+Y\).
Their covariances remain uniformly close and converge to one another away from the diagonal.
Critical-chaos uniqueness first identifies the limiting law and then gives convergence in probability to the same measure on the original coupling. 
This yields the canonical identification in Proposition~\ref{prop:canonical-corrected-chaos}. 
Finally, a constant-mode identity shows that the Gaussian constant mode contributes only an almost surely positive scalar factor.
Removing this factor transfers the Rajchman property to \(M_\phi^{\mathrm{crit}}\) and completes the proof of Theorem~\ref{thm:main-circle-rajchman}.

\subsection{Scope and organization}
\label{subsec:introduction-scope}

Theorem~\ref{thm:main-circle-rajchman} is qualitative. 
The polynomial bound in Proposition~\ref{prop:predictable-annular-decay} applies only to the coarse predictable measure used in the proof
and does not give a decay rate for the terminal coefficients \(\widehat{M_\phi^{\mathrm{crit}}}(n)\). 
Determining their almost-sure asymptotic order remains open. 
The compact-range star-scale field is likewise an auxiliary device:
it provides the finite-range dependence needed for the annular concentration argument,
whereas the theorem concerns the canonical mollifier-independent critical chaos associated with the exact circle covariance \eqref{eq:exact-circle-covariance}.
The covariance correction and critical-chaos uniqueness argument are therefore part of the identification of the measure in the theorem,
rather than a formal change of model.

Section~\ref{sec:periodized-field} constructs the periodized compact-range field,  its positive truncated critical martingales, the derivative limit, and the rooted law. 
Section~\ref{sec:rooted-small-cells} proves the rooted regression estimate, polynomial small-cell summability, and the realized and predictable moving cell tails. 
Section~\ref{sec:terminal-locality-annular} establishes terminal locality, conditional independence of separated cell restrictions,
and the uniform one-scale annular approximation. 
Section~\ref{sec:periodized-rajchman} controls the variation of the coarse predictable measure, proves the Rajchman property for the periodized derivative measure,
and identifies the corresponding positive critical scale measure.
Finally, Section~\ref{sec:exact-circle-transfer} constructs the smooth positive-definite covariance correction,
compares the scale and mollifier approximations, removes the constant Gaussian mode, and proves Theorem~\ref{thm:main-circle-rajchman}.

\section{The periodized compact-range field}
\label{sec:periodized-field}

We first work with a periodized compact-range scale decomposition.  
This model has two features that will be used separately.  
On every sufficiently short arc, its complete space-time law agrees exactly with the real-line compact-range field of Duplantier--Rhodes--Sheffield--Vargas.  
On the other hand, the field added after time \(t\) has spatial dependence range \(e^{-t}\).  
The first property gives the critical construction in this section,  while the second will later produce conditional independence between separated cells.

\subsection{The kernel and the scale field}
\label{subsec:kernel-scale-field}

We use the real-line Fourier transform
\[
\widehat{k}(\xi)
=
\int_{\mathbb{R}}k(x)e^{-i\xi x}\,\mathrm{d}x.
\]

\begin{definition}
\label{def:admissible-compact-range-kernel}
An admissible compact-range kernel is an even function \(k\in C_c^\infty(\mathbb{R})\) satisfying
\[
k(0)=1,
\qquad
\widehat{k}(\xi)\geq0,
\qquad
xk'(x)\leq0,
\qquad
\operatorname{supp}k\subset[-1,1].
\]
\end{definition}

Such kernels exist.
For example, let \(\varphi\in C_c^\infty(\mathbb{R})\) be nonnegative, even, nonconstant, supported in \([-1/4,1/4]\), and nonincreasing on \([0,\infty)\).  
Then, the normalized autocorrelation
\[
k(x)
=
\frac{(\varphi*\widetilde{\varphi})(x)}
     {(\varphi*\widetilde{\varphi})(0)},
\qquad
\widetilde{\varphi}(x)=\varphi(-x),
\]
is admissible.
We fix one such kernel throughout the paper.

For \(u\geq1\), define the periodization
\begin{equation}
\label{eq:periodized-kernel-definition}
q_u(\theta)
=
\sum_{\ell\in\mathbb{Z}}
k\bigl(u(\theta+2\pi\ell)\bigr),
\qquad
\theta\in\mathbb{T}.
\end{equation}

\begin{lemma}
\label{lem:periodized-positive-definite-local-lifting}
For every \(u\geq1\), the function \(q_u\) is smooth and positive definite on \(\mathbb{T}\), and
\[
q_u(0)=1.
\]

Let \(J\subset\mathbb{T}\) be an arc of geodesic diameter less than \(1/2\), and let \(\widetilde{J}\subset\mathbb{R}\) be a lift of \(J\).
If \(\theta,\theta'\in J\) have lifts \(x,x'\in\widetilde{J}\), then
\begin{equation}
\label{eq:exact-local-lifting-kernel}
q_u(\theta-\theta')
=
k\bigl(u(x-x')\bigr).
\end{equation}
\end{lemma}

\begin{proof}
Since \(k\) is smooth and compactly supported, the periodization in \eqref{eq:periodized-kernel-definition} is locally finite, and hence \(q_u\) is smooth. 
Moreover, for every \(n\in\mathbb{Z}\), 
\begin{equation}
\label{eq:periodized-kernel-fourier-coefficient}
\begin{aligned}
\frac{1}{2\pi}
\int_0^{2\pi}
q_u(\theta)e^{-in\theta}\,\mathrm{d}\theta
=
\frac{1}{2\pi}
\sum_{\ell\in\mathbb{Z}}
\int_0^{2\pi}
k\bigl(u(\theta+2\pi\ell)\bigr)
e^{-in\theta}\,\mathrm{d}\theta
=
\frac{1}{2\pi u}
\widehat{k}\left(\frac{n}{u}\right)
\geq0,
\end{aligned}
\end{equation}
which implies that \(q_u\) is positive definite by the Bochner--Herglotz theorem on \(\mathbb{T}\).
At the origin, all periodization terms except \(\ell=0\) vanish, and therefore,
\[
q_u(0)=k(0)=1.
\]

For the local identity, \(|x-x'|<1/2\). 
If \(\ell\neq0\), then
\[
|x-x'+2\pi\ell|
\geq
2\pi-\frac{1}{2}
>
1.
\]
Since \(u\geq1\) and \(k\) is supported in \([-1,1]\), every term with \(\ell\neq0\) in \eqref{eq:periodized-kernel-definition} vanishes.
Thus, only the \(\ell=0\) term remains, which yields \eqref{eq:exact-local-lifting-kernel}.
\end{proof}

For \(t\geq0\), put
\begin{equation}
\label{eq:periodized-covariance-Kt}
K_t(h)
=
\int_1^{e^t}
q_u(h)\,\frac{\mathrm{d}u}{u},
\qquad
h\in\mathbb{T}.
\end{equation}
Let \((X_t)_{t\geq0}\) be the centered Gaussian field with independent increments in the scale variable and covariance
\begin{equation}
\label{eq:periodized-scale-covariance}
\mathbb{E}\left[
X_t(\theta)X_s(\theta')
\right]
=
K_{t\wedge s}(\theta-\theta')
=
\int_1^{e^{t\wedge s}}
q_u(\theta-\theta')\,\frac{\mathrm{d}u}{u}.
\end{equation}
The Gaussian construction follows from Lemma~\ref{lem:periodized-positive-definite-local-lifting}. 
Since \(q_u(0)=1\),
\begin{equation}
\label{eq:periodized-pointwise-variance}
\mathbb{E}[X_t(\theta)^2]=t.
\end{equation}
In particular, for every fixed \(\theta\), the process \((X_t(\theta))_{t\geq0}\) is a standard Brownian motion.

For every finite \(T\), the covariance on \([0,T]\times\mathbb{T}\) is smooth in the spatial variables. 
Standard Gaussian regularity gives a version that is jointly continuous in \((t,\theta)\), and for which \(\theta\mapsto X_t(\theta)\) is smooth at every finite time. 
We use this version throughout.

If \(J\) and \(\widetilde{J}\) are as in Lemma~\ref{lem:periodized-positive-definite-local-lifting}, then
\begin{equation}
\label{eq:exact-local-lifting-covariance}
\mathbb{E}\left[
X_t(\theta)X_s(\theta')
\right]
=
\int_1^{e^{t\wedge s}}
k\bigl(u(x-x')\bigr)\,\frac{\mathrm{d}u}{u}.
\end{equation}
Thus, the complete space-time process restricted to \(J\), after lifting, has exactly the compact-range star-scale law used in \cite{DRSVDerivative}. 
Every application of the real-line critical theory below is made on such a fixed lifted arc; no periodic extension of the cited theorems is assumed.

\subsection{Positive truncations and the rooted law}
\label{subsec:positive-truncations-rooted-law}

Let \((\mathcal{F}_t)_{t\geq0}\) denote the completed natural filtration of the field. 
For \(\beta>0\), define
\[
R_t^\beta(\theta)
=
\beta+\sqrt{2}\,t-X_t(\theta),
\]
and
\[
\tau_\theta^\beta
=
\inf\left\{
s>0:
X_s(\theta)-\sqrt{2}\,s>\beta
\right\},
\]
with the convention \(\inf\varnothing=\infty\). 
Put
\begin{equation}
\label{eq:positive-truncated-density}
F_t^\beta(\theta)
=
R_t^\beta(\theta)
\mathbf{1}_{\{\tau_\theta^\beta>t\}}
e^{\sqrt{2}X_t(\theta)-t},
\end{equation}
and define the positive random measure
\begin{equation}
\label{eq:positive-truncated-measure}
Z_t^\beta(\mathrm{d}\theta)
=
F_t^\beta(\theta)\,\mathrm{d}\theta.
\end{equation}
Here and below, \(|A|\) denotes Lebesgue arc length on \(\mathbb{T}\), so that \(|\mathbb{T}|=2\pi\).

\begin{proposition}
\label{prop:periodized-positive-rooted-package}
Fix \(\beta>0\). 
Then, the following assertions hold.

\begin{enumerate}
\item[(i)] There exists a finite positive random measure \(Z^\beta\) on \(\mathbb{T}\) such that
\begin{equation}
\label{eq:positive-truncation-weak-limit}
Z_t^\beta
\xrightarrow[t\to\infty]{\mathrm{w}}
Z^\beta
\qquad
\text{almost surely}.
\end{equation}

\item[(ii)] For every bounded Borel function \(f\colon\mathbb{T}\to\mathbb{C}\), write
\[
Z_t^\beta(f)
=
\int_{\mathbb{T}}
f(\theta)\,Z_t^\beta(\mathrm{d}\theta),
\qquad
Z^\beta(f)
=
\int_{\mathbb{T}}
f(\theta)\,Z^\beta(\mathrm{d}\theta).
\]
Then,
\begin{equation}
\label{eq:positive-truncation-first-moment}
\mathbb{E}\bigl[Z^\beta(f)\bigr]
=
\beta
\int_{\mathbb{T}}
f(\theta)\,\mathrm{d}\theta,
\end{equation}
and
\begin{equation}
\label{eq:positive-truncation-martingale-closure}
\mathbb{E}\left[
Z^\beta(f)\mid\mathcal{F}_t
\right]
=
Z_t^\beta(f)
\qquad
\text{almost surely}.
\end{equation}

\item[(iii)] The measure
\begin{equation}
\label{eq:terminal-rooted-law}
\Theta^\beta(\mathrm{d}\theta,\mathrm{d}\omega)
:=
\frac{1}{2\pi\beta}
Z_\omega^\beta(\mathrm{d}\theta)
\mathbb{P}(\mathrm{d}\omega)
\end{equation}
is a probability measure. 
For every \(t\geq0\), its restriction to \(\mathcal{B}(\mathbb{T})\otimes\mathcal{F}_t\) is
\begin{equation}
\label{eq:finite-rooted-law}
\Theta_t^\beta(\mathrm{d}\theta,\mathrm{d}\omega)
:=
\frac{1}{2\pi\beta}
Z_{t,\omega}^\beta(\mathrm{d}\theta)
\mathbb{P}(\mathrm{d}\omega).
\end{equation}
Under \(\Theta^\beta\), the root \(\theta\) is uniform on \(\mathbb{T}\),
and, conditionally on the root, \( \bigl(R_s^\beta(\theta)\bigr)_{s\geq0} \) is a three-dimensional Bessel process started from \(\beta\).
\end{enumerate}
\end{proposition}

\begin{proof}
We first establish the martingale and uniform-integrability input used below. 
For fixed \(\theta\in\mathbb{T}\), \((X_t(\theta))_{t\geq0}\) is a standard Brownian motion. 
The usual one-point Girsanov--stopping argument, as in \cite[proof of Proposition~13]{DRSVDerivative}, yields
\begin{equation}
\label{eq:pointwise-positive-martingale}
\mathbb{E}\left[
F_t^\beta(\theta)\mid\mathcal{F}_s
\right]
=
F_s^\beta(\theta),
\qquad
0\leq s\leq t.
\end{equation}
Indeed, after the exponential change of measure associated with \( e^{\sqrt{2}(X_t(\theta)-X_s(\theta))-(t-s)} \),
the future barrier process is Brownian motion started from \(R_s^\beta(\theta)\) and killed at zero, and the required conditional mean follows from optional stopping. 
Conditional Fubini then shows that \(Z_t^\beta(f)\) is a martingale for every bounded Borel \(f\colon\mathbb{T}\to\mathbb{C}\).

We next prove uniform integrability. 
Cover \(\mathbb{T}\) by finitely many open arcs \(J_1,\ldots,J_L\), each of diameter less than \(1/2\).
By \eqref{eq:exact-local-lifting-covariance}, after lifting \(J_j\) to the real line,
the complete space-time field has exactly the compact-range law considered in \cite{DRSVDerivative}. 
The present kernel satisfies the hypotheses of \cite[Proposition~14]{DRSVDerivative}; hence \( (Z_t^\beta(J_j))_{t\geq0} \) is uniformly integrable for every \(j\). 
Since
\[
0\leq
Z_t^\beta(\mathbb{T})
\leq
\sum_{j=1}^L Z_t^\beta(J_j),
\]
the total masses form a uniformly integrable family. 
Therefore,
\[
|Z_t^\beta(f)|
\leq
\|f\|_{L^\infty(\mathbb{T})}
Z_t^\beta(\mathbb{T})
\]
gives uniform integrability of \((Z_t^\beta(f))_{t\geq0}\) for every bounded Borel \(f\).

Choose a countable uniformly dense subset \(\mathcal{D}\subset C(\mathbb{T})\) containing the constant function \(1\). 
On a common probability-one event, \(Z_t^\beta(f)\) converges as \(t\to\infty\) for every \(f\in\mathcal{D}\),
and \(Z_t^\beta(\mathbb{T})\) remains bounded for all sufficiently large \(t\). 
Since \(\mathbb{T}\) is compact, every sequence \(t_n\to\infty\) has a subsequence along which \(Z_{t_n}^\beta\) converges weakly to a finite positive measure.
Any two such subsequential limits agree on \(\mathcal{D}\), and hence on all of \(C(\mathbb{T})\). 
Thus, the weak limit is unique, which proves \eqref{eq:positive-truncation-weak-limit}.

Let \(f\in C(\mathbb{T})\). 
By \eqref{eq:positive-truncation-weak-limit},
\[
Z_t^\beta(f)\longrightarrow Z^\beta(f)
\qquad
\text{almost surely}.
\]
Uniform integrability identifies this limit with the terminal value of the martingale, so
\begin{equation}
\label{eq:continuous-test-martingale-closure}
\mathbb{E}\left[
Z^\beta(f)\mid\mathcal{F}_t
\right]
=
Z_t^\beta(f).
\end{equation}
Since \( Z_0^\beta(\mathrm{d}\theta) = \beta\,\mathrm{d}\theta \), we also have
\begin{equation}
\label{eq:continuous-test-first-moment}
\mathbb{E}\bigl[Z^\beta(f)\bigr]
=
\beta
\int_{\mathbb{T}}
f(\theta)\,\mathrm{d}\theta.
\end{equation}
The two sides of \eqref{eq:continuous-test-first-moment} define finite Borel measures that agree on \(C(\mathbb{T})\). 
They therefore agree as measures, which proves \eqref{eq:positive-truncation-first-moment} for every bounded Borel \(f\).
Similarly, fix \(t\geq0\) and \(G\in\mathcal{F}_t\), and define the finite Borel measures
\[
\mu_G(A)
=
\mathbb{E}\left[
\mathbf{1}_G Z^\beta(A)
\right],
\qquad
\nu_G(A)
=
\mathbb{E}\left[
\mathbf{1}_G Z_t^\beta(A)
\right].
\]
By \eqref{eq:continuous-test-martingale-closure},
\[
\mu_G=\nu_G.
\]
Since \(G\in\mathcal{F}_t\) is arbitrary, \eqref{eq:positive-truncation-martingale-closure} follows.

Equation \eqref{eq:positive-truncation-first-moment} shows that \(\Theta^\beta\) is a probability measure with spatial marginal \(\mathrm{d}\theta/(2\pi)\).  
Moreover, for every Borel \(A\subset\mathbb{T}\) and \(G\in\mathcal{F}_t\), by \eqref{eq:positive-truncation-martingale-closure},
\[
\Theta^\beta(A\times G)
=
\frac{1}{2\pi\beta}
\mathbb{E}\left[
\mathbf{1}_G Z_t^\beta(A)
\right],
\]
which proves \eqref{eq:finite-rooted-law}.

Finally, conditionally on the root \(\theta\), the Radon--Nikodym density of \(\Theta_t^\beta\) with respect to \(\mathbb{P}\) is
\[
\frac{1}{\beta}
R_t^\beta(\theta)
\mathbf{1}_{\{\tau_\theta^\beta>t\}}
e^{\sqrt{2}X_t(\theta)-t}.
\]
Under the exponential tilt,  \(R^\beta(\theta)\) is Brownian motion started from \(\beta\), and the factor
\[
\frac{R_t^\beta(\theta)}{\beta}
\mathbf{1}_{\{\tau_\theta^\beta>t\}}
\]
is precisely its Doob \(h\)-transform density, with \(h(x)=x\), for Brownian motion killed at zero.  
Hence, the transformed process is \(\operatorname{BES}(3)\) started from \(\beta\);
this is the rooted change of measure used in \cite[proof of Proposition~14]{DRSVDerivative}.  
Since this identification holds on every finite horizon, the complete rooted path under \(\Theta^\beta\) is \(\operatorname{BES}(3)\) started from \(\beta\).
\end{proof}

\subsection{The global barrier and derivative limit}
\label{subsec:global-barrier-derivative-limit}

Define the ordinary critical measures
\begin{equation}
\label{eq:ordinary-critical-measure}
M_t^{\sqrt{2}}(\mathrm{d}\theta)
=
e^{\sqrt{2}X_t(\theta)-t}\,\mathrm{d}\theta,
\end{equation}
and the signed derivative approximations
\begin{equation}
\label{eq:derivative-approximation}
M_t'(\mathrm{d}\theta)
=
\bigl(\sqrt{2}\,t-X_t(\theta)\bigr)
e^{\sqrt{2}X_t(\theta)-t}\,\mathrm{d}\theta.
\end{equation}
Set
\begin{equation}
\label{eq:global-barrier-supremum}
S
=
\sup_{t\geq0}
\sup_{\theta\in\mathbb{T}}
\bigl(X_t(\theta)-\sqrt{2}\,t\bigr),
\end{equation}
and, for \(\beta>0\),
\begin{equation}
\label{eq:global-barrier-event}
E_\beta
=
\{S<\beta\}.
\end{equation}

Choose a finite cover of \(\mathbb{T}\) by open arcs of diameter less than \(1/2\). 
By \eqref{eq:exact-local-lifting-covariance}, the field on each lifted arc has exactly the compact-range real-line law of \cite{DRSVDerivative}.  
Applying \cite[Proposition~19]{DRSVDerivative} on each lifted arc and then using the finite cover, we obtain that almost surely
\begin{equation}\label{eq:ordinary-critical-total-mass-vanishing}
M_t^{\sqrt{2}}(\mathbb{T}) \longrightarrow 0
\qquad\text{and}\qquad
S<\infty.
\end{equation}

\begin{proposition}
\label{prop:global-barrier-derivative-limit}
There exists a positive finite random measure \(M_*'\) on \(\mathbb{T}\), with full support, such that
\begin{equation}
\label{eq:global-derivative-weak-limit}
M_t'
\xrightarrow[t\to\infty]{\mathrm{w}}
M_*'
\qquad
\text{almost surely}.
\end{equation}
For every deterministic \(\beta>0\),
\begin{equation}
\label{eq:barrier-identification}
M_*'=Z^\beta
\qquad
\text{on }E_\beta.
\end{equation}
\end{proposition}

\begin{proof}
Fix an integer \(j\geq1\). 
On \(E_j\), the barrier is never crossed, so
\begin{equation}
\label{eq:derivative-truncation-identity}
Z_t^j
=
M_t'+jM_t^{\sqrt{2}},
\qquad t\geq0.
\end{equation}
Proposition~\ref{prop:periodized-positive-rooted-package} and \eqref{eq:ordinary-critical-total-mass-vanishing} therefore imply
\begin{equation}
\label{eq:derivative-limit-on-integer-barrier}
M_t'
\xrightarrow[t\to\infty]{\mathrm{w}}
Z^j
\qquad
\text{on }E_j.
\end{equation}
On \(E_j\cap E_k\), both \(Z^j\) and \(Z^k\) are weak limits of the same sequence \((M_t')_{t\geq0}\), and therefore agree. 
With \(E_0=\varnothing\), define
\begin{equation}
\label{eq:patched-derivative-limit}
M_*'
=
\sum_{j=1}^{\infty}
\mathbf{1}_{E_j\setminus E_{j-1}}Z^j.
\end{equation}
Exactly one term is present almost surely.  
The resulting random measure is positive and finite, and \eqref{eq:derivative-limit-on-integer-barrier} yields \eqref{eq:global-derivative-weak-limit}.

Now fix a deterministic \(\beta>0\). 
On \(E_\beta\),
\[
Z_t^\beta
=
M_t'+\beta M_t^{\sqrt{2}}.
\]
Passing to the weak limits in this identity proves \eqref{eq:barrier-identification}.

It remains to establish full support. 
Let \(\mathcal{U}\) be a countable basis of open arcs of diameter less than \(1/2\). 
For each \(U\in\mathcal{U}\), choose a nonzero function \(\varphi_U\in C(\mathbb{T})\) with
\[
0\leq\varphi_U\leq1,
\qquad
\operatorname{supp}\varphi_U\subset U.
\]
Exact local lifting and \cite[Theorem~17]{DRSVDerivative} imply that \(M_t'(\varphi_U)\) converges almost surely to a strictly positive limit. 
By \eqref{eq:global-derivative-weak-limit}, this limit is \(M_*'(\varphi_U)\). 
Intersecting the corresponding probability-one events over the countable family \(\mathcal{U}\), we obtain
\[
M_*'(U)>0
\qquad
\text{for every }U\in\mathcal{U}.
\]
Hence, \(M_*'\) has full support almost surely.
\end{proof}

We do not use the atomlessness conclusion of \cite[Theorem~17]{DRSVDerivative} here. 
Atomlessness will instead follow from the stronger rooted small-cell summability theorem in Section~\ref{sec:rooted-small-cells},
which is also needed for the moving cell tails.

\section{Rooted nonconcentration and moving cell tails}
\label{sec:rooted-small-cells}

This section proves the nonconcentration estimate that will later make the one-scale Fourier argument possible. 
Under the rooted law, the barrier process is a three-dimensional Bessel process. 
An exact Gaussian regression against the rooted path then shows that the conditional density at a nearby point contains a strictly negative displacement. 
This yields summable small-cell bounds at a typical root and, in turn, the realized and predictable moving mass tails needed in the annular concentration argument.

For \(m\geq1\) and \(0\leq j<2^m\), define
\[
I_j^{(m)}
=
\left[
\frac{2\pi j}{2^m},
\frac{2\pi(j+1)}{2^m}
\right),
\]
with endpoints interpreted on \(\mathbb{T}\), and set
\[
\mathcal{P}_m
=
\left\{
I_j^{(m)}:
0\leq j<2^m
\right\}.
\]
Thus, \(\mathcal{P}_m\) is the dyadic partition of \(\mathbb{T}\) into half-open arcs of equal length \( \ell_m=2\pi2^{-m}. \)
For \(\theta\in\mathbb{T}\), let \(Q_m(\theta)\) denote the unique cell of \(\mathcal{P}_m\) containing \(\theta\).

\subsection{Bessel envelopes and Gaussian regression}
\label{subsec:bessel-envelopes-regression}

Throughout the remainder of this section, \(\beta>0\) denotes a fixed
deterministic parameter unless explicitly quantified otherwise.

For \(T\geq0\), let
\[
\mathcal{H}_T
=
\sigma\left(
\theta,
\bigl(X_s(\theta)\bigr)_{0\leq s\leq T}
\right)
\]
be the sigma-algebra generated by the root and its path up to time \(T\).
Let
\[
\mathbb{Q}_0
=
\frac{\mathrm{d}\theta}{2\pi}\otimes\mathbb{P}
\]
be the product of uniform measure on the circle and the original environment law. 
By \eqref{eq:finite-rooted-law}, the Radon--Nikodym derivative of \(\Theta_T^\beta\) with respect to \(\mathbb{Q}_0\) is \( \frac{1}{\beta}F_T^\beta(\theta), \)
which is \(\mathcal{H}_T\)-measurable. 
Consequently, for every random variable \(U\) that is integrable under both \(\Theta_T^\beta\) and \(\mathbb{Q}_0\),
\begin{equation}
\label{eq:rooted-conditional-kernel-cancellation}
\mathbb{E}_{\Theta_T^\beta}
\left[
U\mid\mathcal{H}_T
\right]
=
\mathbb{E}_{\mathbb{Q}_0}
\left[
U\mid\mathcal{H}_T
\right]
\qquad
\Theta_T^\beta\text{-almost surely}.
\end{equation}
Thus, rooting changes the law of the root path but does not change the conditional Gaussian law of the field away from the root.

For \(R\geq1\) and \(T\geq0\), define
\begin{equation}
\label{eq:Bessel-envelope-event}
\mathcal{B}_R(T)
=
\biggl\{
\frac{\sqrt{s}}{R\log^2(2+s)}
\leq
R_s^\beta(\theta)
\leq
R\left(
1+\sqrt{s\log(1+s)}
\right),\quad
\forall\ 0\leq s\leq T
\biggr\}.
\end{equation}
Put
\[
\mathcal{B}_R(\infty)
=
\bigcap_{T\geq0}\mathcal{B}_R(T).
\]
By Proposition~\ref{prop:periodized-positive-rooted-package}, the standard upper and lower Bessel envelopes in \cite[Theorem~15]{DRSVDerivative},
together with continuity and strict positivity on bounded time intervals, imply that
\begin{equation}
\label{eq:Bessel-envelope-exhaustion}
\Theta^\beta\left(
\bigcup_{R=1}^{\infty}
\mathcal{B}_R(\infty)
\right)
=
1.
\end{equation}

Fix a root \(\theta\) and \(w\in\mathbb{T}\) with
\[
0<r=d_{\mathbb{T}}(\theta,w)<\frac{1}{2}.
\]
Choose lifts whose signed difference is \(h\in(-1/2,1/2)\), so that \(|h|=r\).  
For \(s\geq0\), define
\[
a_s(h)=k(e^sh),
\qquad
\alpha_s(h)=-e^sh\,k'(e^sh).
\]
Then,
\[
\frac{\mathrm d}{\mathrm ds}a_s(h)
=
-\alpha_s(h),
\]
and the monotonicity assumption on \(k\) gives \( \alpha_s(h)\geq0. \)
Also set
\[
A_s(h)
=
\int_0^s a_u(h)\,\mathrm{d}u,
\]
and define the Gaussian projection
\begin{equation}
\label{eq:rooted-Gaussian-projection}
P_s^h
=
\int_0^s
a_u(h)\,\mathrm{d}X_u(\theta).
\end{equation}

\begin{lemma}
\label{lem:exact-root-gaussian-regression}
Let
\[
\mathcal{Z}_s^h
=
X_s(w)-P_s^h.
\]
Then, \((\mathcal{Z}_s^h)_{s\geq0}\) is a centered Gaussian process.
It is independent of \(\bigl(X_t(\theta)\bigr)_{t\geq0}\), and
\begin{equation}
\label{eq:rooted-residual-covariance}
\mathbb{E}\left[
\mathcal{Z}_s^h\mathcal{Z}_{s'}^h
\right]
=
s\wedge s'
-
\int_0^{s\wedge s'}
a_u(h)^2\,\mathrm{d}u.
\end{equation}
Moreover, there is a constant \(C_{\beta,k}<\infty\) such that, whenever \( s\leq\log(1/r), \) one has
\begin{equation}
\label{eq:rooted-residual-variance-bound}
0\leq
\operatorname{Var}(\mathcal{Z}_s^h)
\leq C_{\beta,k},
\end{equation}
and
\begin{equation}
\label{eq:rooted-drift-correction-bound}
\left|
\sqrt{2}\bigl(s-A_s(h)\bigr)-\beta k(h)
\right|
\leq C_{\beta,k}.
\end{equation}
\end{lemma}

\begin{proof}
Since \(X\) is a centered Gaussian field and \(P_s^h\) is a deterministic linear Gaussian functional of the rooted path, 
\((\mathcal Z_s^h)_{s\geq0}\) is a centered Gaussian process. 
Since \((X_t(\theta))_{t\geq0}\) is a standard Brownian motion, the Itô covariance identity and \eqref{eq:exact-local-lifting-covariance} yield, for every \(t\geq0\),
\[
\mathbb{E}\left[
P_s^hX_t(\theta)
\right]
=
\int_0^{s\wedge t}
a_u(h)\,\mathrm{d}u
=
\mathbb{E}\left[
X_s(w)X_t(\theta)
\right].
\]
Thus, \(\mathcal{Z}_s^h\) is orthogonal to the complete rooted path.
Joint Gaussianity therefore implies independence.  
Subtracting the covariance of the projection yields \eqref{eq:rooted-residual-covariance}.

If \(s\leq\log(1/r)\), then \(e^sr\leq1\), and evenness of \(k\) implies
\[
\operatorname{Var}(\mathcal{Z}_s^h)
=
\int_0^s
\left(
1-k(e^ur)^2
\right)\,\mathrm{d}u
\leq
\int_0^1
\frac{1-k(v)^2}{v}\,\mathrm{d}v.
\]
Since \(k\) is even and smooth with \(k(0)=1\), this proves \eqref{eq:rooted-residual-variance-bound}.  
The estimate \eqref{eq:rooted-drift-correction-bound} follows analogously, with \(1-k(v)\) in place of \(1-k(v)^2\), together with \(0\leq k(h)\leq1\).
\end{proof}

The following identity records the sign structure behind the rooted estimate.

\begin{lemma}
\label{lem:rooted-projected-displacement}
Suppose that \( s\leq\log(1/r). \)
Define
\begin{equation}
\label{eq:rooted-negative-displacement}
\Delta_s^h
=
-\int_0^s
\alpha_u(h)R_u^\beta(\theta)\,\mathrm{d}u
-
a_s(h)R_s^\beta(\theta).
\end{equation}
Then,
\begin{equation}
\label{eq:rooted-displacement-projection-identity}
\Delta_s^h
=
P_s^h
-
\sqrt{2}A_s(h)
-
\beta k(h).
\end{equation}
Equivalently, with
\[
\vartheta_h(s)
=
\sqrt{2}\bigl(s-A_s(h)\bigr)-\beta k(h),
\]
one has
\begin{equation}
\label{eq:rooted-projection-centered-form}
P_s^h
=
\sqrt{2}\,s+\Delta_s^h-\vartheta_h(s).
\end{equation}

Fix \(R\geq1\). 
There are constants \(c_R>0\), \(C_R<\infty\), and \(s_R<\infty\), depending also on \(\beta\) and \(k\), such that on \(\mathcal{B}_R(T)\), whenever
\[
s_R\leq s\leq T
\qquad\text{and}\qquad
s\leq\log\frac{1}{r},
\]
one has
\begin{equation}
\label{eq:rooted-displacement-two-sided}
-C_R\left(
1+\sqrt{s\log(1+s)}
\right)
\leq
\Delta_s^h
\leq
-c_R\frac{\sqrt{s}}{\log^2(2+s)}.
\end{equation}
\end{lemma}

\begin{proof}
Stochastic integration by parts gives
\[
P_s^h
=
a_s(h)X_s(\theta)
+
\int_0^s
\alpha_u(h)X_u(\theta)\,\mathrm du.
\]
Using
\[
X_u(\theta)
=
\beta+\sqrt{2}\,u-R_u^\beta(\theta),
\]
one obtains \eqref{eq:rooted-displacement-projection-identity}.

Choose \(L>0\) such that \( k(e^{-L})\geq1/2. \)
Since \(e^sr\leq1\),
\[
a_s(h)
+
\int_{s-L}^s
\alpha_u(h)\,\mathrm{d}u
=
a_{s-L}(h)
\geq
k(e^{-L})
\geq\frac{1}{2}
\]
whenever \(s\geq L\).  
On \(\mathcal{B}_R(T)\), the lower Bessel envelope is comparable throughout the fixed interval \([s-L,s]\), and therefore
\[
-\Delta_s^h
\geq
c_R\frac{\sqrt{s}}{\log^2(2+s)}
\]
for all sufficiently large \(s\).
The upper bound in \eqref{eq:rooted-displacement-two-sided} follows similarly from the upper Bessel envelope and the fact that the total weight is bounded by \(1\). 
Enlarging the constants if necessary to cover bounded \(s\) completes the proof.
\end{proof}

\subsection{Finite-horizon rooted density estimates}
\label{subsec:finite-horizon-rooted-density}
For fixed \(\beta>0\) and \(R\geq1\), let
\begin{equation}
\label{eq:rooted-density-majorant}
G_{\beta,R}(s)
=
C_{\beta,R,k}
\left(
1+s\log(2+s)
\right)
\exp\left(
-c_{\beta,R,k}
\frac{\sqrt{s}}{\log^2(2+s)}
\right),
\qquad
s\geq0,
\end{equation}
where \(C_{\beta,R,k}<\infty\) and \(c_{\beta,R,k}>0\) are chosen sufficiently large and small, respectively,
in terms of the constants in Lemma~\ref{lem:rooted-projected-displacement}. 
Since
\[
\frac{\sqrt{s}}{\log^2(2+s)}
\geq
s^{1/3}
\]
for all sufficiently large \(s\), one has
\begin{equation}
\label{eq:rooted-density-majorant-moments}
\int_0^\infty
(1+s)^aG_{\beta,R}(s)\,\mathrm{d}s
<
\infty
\qquad
\text{for every }a\geq0.
\end{equation}

\begin{lemma}
\label{lem:finite-horizon-rooted-density}
Let \(0\leq T<\infty\), and let
\[
r=d_{\mathbb{T}}(\theta,w)<\frac{1}{2}.
\]
If \( e^{-T}<r<1/2, \) then
\begin{equation}
\label{eq:separated-rooted-density}
\mathbf{1}_{\mathcal{B}_R(T)}
\mathbb{E}_{\Theta_T^\beta}
\left[
F_T^\beta(w)\mid\mathcal{H}_T
\right]
\leq
\frac{1}{r}
G_{\beta,R}\left(
\log\frac{1}{r}
\right).
\end{equation}
If \( 0<r\leq e^{-T}, \) then
\begin{equation}
\label{eq:near-horizon-rooted-density}
\mathbf{1}_{\mathcal{B}_R(T)}
\mathbb{E}_{\Theta_T^\beta}
\left[
F_T^\beta(w)\mid\mathcal{H}_T
\right]
\leq
e^T G_{\beta,R}(T).
\end{equation}
\end{lemma}

\begin{proof}
By \eqref{eq:rooted-conditional-kernel-cancellation}, all conditional Gaussian calculations may be performed under \(\mathbb{Q}_0\).

Suppose first that
\[
e^{-T}<r<\frac{1}{2},
\qquad
s=\log\frac{1}{r}<T.
\]
For \(u\geq s\),
\[
a_u(h)=k(e^uh)=0.
\]
The field increments at \(w\) after time \(s\) are therefore independent of the complete rooted path and of the field up to time \(s\).
Consequently, the pointwise martingale identity \eqref{eq:pointwise-positive-martingale} remains valid after enlarging the conditioning sigma-algebra from \(\mathcal{F}_s\) to  \(\mathcal{F}_s\vee\mathcal{H}_T\):
\[
\mathbb{E}_{\mathbb{Q}_0}
\left[
F_T^\beta(w)\mid\mathcal{F}_s\vee\mathcal{H}_T
\right]
=
F_s^\beta(w).
\]
Taking conditional expectation with respect to \(\mathcal{H}_T\) then gives
\begin{equation}
\label{eq:separated-martingale-cut}
\mathbb{E}_{\mathbb{Q}_0}
\left[
F_T^\beta(w)\mid\mathcal{H}_T
\right]
=
\mathbb{E}_{\mathbb{Q}_0}
\left[
F_s^\beta(w)\mid\mathcal{H}_T
\right].
\end{equation}
In the near-horizon regime, take instead \(s=T\).
In both cases, \(s\leq\log(1/r)\), and Lemma~\ref{lem:exact-root-gaussian-regression} applies.

Write \( X_s(w)=P_s^h+\mathcal{Z}_s^h. \)
By \eqref{eq:rooted-projection-centered-form},
\[
\beta+\sqrt{2}\,s-P_s^h
=
\beta-\Delta_s^h+\vartheta_h(s),
\]
where \(\vartheta_h(s)\) is uniformly bounded by \eqref{eq:rooted-drift-correction-bound}.
On the survival event at \(w\),
\[
R_s^\beta(w)
=
\beta+\sqrt{2}\,s-P_s^h-\mathcal{Z}_s^h
\geq0.
\]
Using \( x\mathbf{1}_{\{x\geq0\}} \leq 1+x^2, \) and dropping the path-survival constraint, we obtain
\[
\mathbb{E}_{\mathbb{Q}_0}
\left[
F_s^\beta(w)\mid\mathcal{H}_T
\right]
\leq
e^{\sqrt{2}P_s^h-s}
\mathbb{E}\left[
\left(
1+
\left(
\beta+\sqrt{2}\,s-P_s^h-\mathcal{Z}_s^h
\right)^2
\right)
e^{\sqrt{2}\mathcal{Z}_s^h}
\right].
\]
The residual variance is uniformly bounded by \eqref{eq:rooted-residual-variance-bound}. 
Elementary Gaussian moment identities, together with the boundedness of \(\vartheta_h(s)\), therefore yield
\begin{equation}
\label{eq:rooted-density-endpoint-Gaussian-bound}
\mathbb{E}_{\mathbb{Q}_0}
\left[
F_s^\beta(w)\mid\mathcal{H}_T
\right]
\leq
C_{\beta,k}
e^s
\left(
1+(\Delta_s^h)^2
\right)
e^{\sqrt{2}\Delta_s^h}.
\end{equation}

On \(\mathcal{B}_R(T)\), Lemma~\ref{lem:rooted-projected-displacement}
and the choice of the constants in \eqref{eq:rooted-density-majorant} imply that the right-hand side of \eqref{eq:rooted-density-endpoint-Gaussian-bound} is bounded by \( e^sG_{\beta,R}(s). \)
In the separated regime, \( e^s=1/r, \) and \eqref{eq:separated-martingale-cut} yields \eqref{eq:separated-rooted-density}. 
In the near-horizon regime, \(s=T\), which yields \eqref{eq:near-horizon-rooted-density}.
\end{proof}

\subsection{Rooted small-cell summability and moving cell tails}
\label{subsec:rooted-small-cell-summability}

\begin{theorem}\label{thm:rooted-small-cell-summability}
Fix \(\beta>0\) and \(p\geq0\). 
Almost surely,
\begin{equation}
\label{eq:rooted-polynomial-small-cell-summability}
\sum_{m=1}^{\infty}
m^pZ^\beta(Q_m(\theta))
<
\infty,
\qquad
\text{for }Z^\beta\text{-almost every }\theta.
\end{equation}
In particular,
\begin{equation}
\label{eq:rooted-polynomial-small-cell-decay}
m^pZ^\beta(Q_m(\theta))
\longrightarrow0,
\qquad
\text{for }Z^\beta\text{-almost every }\theta.
\end{equation}
\end{theorem}

\begin{proof}
It suffices to prove \eqref{eq:rooted-polynomial-small-cell-summability}, as \eqref{eq:rooted-polynomial-small-cell-decay} is then immediate.
Choose \(m_0\) so large that
\[
\ell_m=2\pi2^{-m}<\frac{1}{2}
\qquad
\text{for }m\geq m_0,
\]
and put
\[
u_m
=
\log\frac{1}{\ell_m}
=
m\log2-\log(2\pi).
\]
Fix \(R\geq1\), \(m\geq m_0\), and \(T\geq u_m\). 
For a root \(\theta\), split \(Q_m(\theta)\) into
\[
d_{\mathbb{T}}(\theta,w)\leq e^{-T}
\qquad
\text{and}
\qquad
e^{-T}<d_{\mathbb{T}}(\theta,w)\leq\ell_m.
\]
Conditional Fubini and Lemma~\ref{lem:finite-horizon-rooted-density} give
\begin{equation}
\label{eq:finite-horizon-cell-bound}
\mathbf{1}_{\mathcal{B}_R(T)}
\mathbb{E}_{\Theta_T^\beta}
\left[
Z_T^\beta(Q_m(\theta))
\mid\mathcal{H}_T
\right]
\leq
C_{\beta,R,k}
\int_{u_m}^{T}
G_{\beta,R}(s)\,\mathrm{d}s
+
C_{\beta,R,k}G_{\beta,R}(T).
\end{equation}
Indeed, the separated estimate is integrated using
\[
s=\log\frac{1}{d_{\mathbb{T}}(\theta,w)},
\]
while the near region has length \(O(e^{-T})\), which cancels the factor \(e^T\) in \eqref{eq:near-horizon-rooted-density}.

Define
\begin{equation}
\label{eq:rooted-cell-deterministic-majorant}
h_{\beta,R}(m)
=
C_{\beta,R,k}
\int_{u_m}^{\infty}
G_{\beta,R}(s)\,\mathrm{d}s.
\end{equation}
By \eqref{eq:rooted-density-majorant-moments},
\begin{equation}
\label{eq:rooted-cell-majorant-summability}
\begin{aligned}
\sum_{m=m_0}^{\infty}
m^ph_{\beta,R}(m)
&\leq
C_{\beta,R,k}
\int_0^\infty
G_{\beta,R}(s)
\sum_{\substack{m\geq m_0\\u_m\leq s}}
m^p\,\mathrm{d}s
\\
&\leq
C_{\beta,R,k,p}
\int_0^\infty
(1+s)^{p+1}G_{\beta,R}(s)\,\mathrm{d}s
<\infty.
\end{aligned}
\end{equation}

Taking \(\Theta_T^\beta\)-expectations in \eqref{eq:finite-horizon-cell-bound} and using the tower property,
the left-hand side becomes \( \mathbb{E}_{\Theta_T^\beta} [ \mathbf{1}_{\mathcal{B}_R(T)} Z_T^\beta(Q_m(\theta)) ] \).
The integrand is \(\mathcal{B}(\mathbb{T})\otimes\mathcal{F}_T\)-measurable,
so the projective identity in Proposition~\ref{prop:periodized-positive-rooted-package} identifies this expectation with the corresponding \(\Theta^\beta\)-expectation.
The environment marginal of \(\Theta^\beta\) is absolutely continuous with respect to \(\mathbb{P}\). 
Hence, the almost-sure weak convergence
\[
Z_T^\beta\xrightarrow{\mathrm{w}}Z^\beta
\]
also holds \(\Theta^\beta\)-almost surely.

By \eqref{eq:positive-truncation-first-moment}, \(Z^\beta\) almost surely charges none of the countably many deterministic partition endpoints.
By the same absolute continuity, this remains true \(\Theta^\beta\)-almost surely. 
Hence, every cell in every \(\mathcal{P}_m\) is a \(Z^\beta\)-continuity set, and therefore
\[
Z_T^\beta(Q_m(\theta))
\longrightarrow
Z^\beta(Q_m(\theta))
\]
for \(\Theta^\beta\)-almost every rooted environment.

Since \( \mathcal{B}_R(T)\downarrow\mathcal{B}_R(\infty) \) and \( G_{\beta,R}(T)\longrightarrow0, \) Fatou's lemma gives
\begin{equation}
\label{eq:terminal-rooted-cell-expectation}
\mathbb{E}_{\Theta^\beta}
\left[
\mathbf{1}_{\mathcal{B}_R(\infty)}
Z^\beta(Q_m(\theta))
\right]
\leq
h_{\beta,R}(m).
\end{equation}
Tonelli's theorem and \eqref{eq:rooted-cell-majorant-summability} now yield
\[
\mathbb{E}_{\Theta^\beta}
\left[
\mathbf{1}_{\mathcal{B}_R(\infty)}
\sum_{m=m_0}^{\infty}
m^pZ^\beta(Q_m(\theta))
\right]
<
\infty.
\]
Using \eqref{eq:Bessel-envelope-exhaustion}, we conclude that
\[
\sum_{m=1}^{\infty}
m^pZ^\beta(Q_m(\theta))
<
\infty
\]
for \(\Theta^\beta\)-almost every rooted environment.

Finally,
\[
0
=
\Theta^\beta\left(
\sum_{m=1}^{\infty}
m^pZ^\beta(Q_m(\theta))
=
\infty
\right)
=
\frac{1}{2\pi\beta}
\mathbb{E}\left[
\int_{\mathbb{T}}
\mathbf{1}_{\left\{
\sum_m m^pZ^\beta(Q_m(\theta))=\infty
\right\}}
Z^\beta(\mathrm{d}\theta)
\right].
\]
The integrand is nonnegative, which proves \eqref{eq:rooted-polynomial-small-cell-summability}.
\end{proof}

\begin{corollary}
\label{cor:atomlessness}
For every \(\beta>0\), the measure \(Z^\beta\) is almost surely atomless.
Consequently, the global derivative limit \(M_*'\) is almost surely atomless.
\end{corollary}

\begin{proof}
Apply Theorem~\ref{thm:rooted-small-cell-summability} with \(p=0\). 
If \(Z^\beta(\{\theta\})>0\), then
\[
Z^\beta(Q_m(\theta))
\geq
Z^\beta(\{\theta\})
\qquad
\text{for every }m,
\]
and hence
\[
\sum_{m=1}^{\infty}
Z^\beta(Q_m(\theta))
=
\infty.
\]
Thus, every atom belongs to the exceptional set in \eqref{eq:rooted-polynomial-small-cell-summability}, which has zero \(Z^\beta\)-mass. 
Therefore, \(Z^\beta\) has no atoms.

By Proposition~\ref{prop:global-barrier-derivative-limit}, almost surely \(M_*'=Z^j\) for some integer \(j\), so \(M_*'\) is also atomless.
\end{proof}

The next consequence of Theorem~\ref{thm:rooted-small-cell-summability} provides the moving-cell tail control needed for the annular concentration argument,
with the partition scale allowed to vary with \(N\).

\begin{proposition}[Moving cell tails]
\label{prop:moving-cell-tails}
Fix a deterministic \(\beta>0\). Let \((m_N)_{N\geq1}\) be a nondecreasing sequence of positive integers such that \( m_N/N\longrightarrow1, \) and put \( t_N=m_N\log2. \)
For every \(b>0\), almost surely,
\begin{equation}
\label{eq:realized-moving-cell-tail}
T_N^\beta(b)
:=
\sum_{I\in\mathcal{P}_{m_N}}
Z^\beta(I)
\mathbf{1}_{\{Z^\beta(I)>b/N\}}
\longrightarrow0,
\end{equation}
and
\begin{equation}
\label{eq:predictable-moving-cell-tail}
B_N^\beta(b)
:=
\sum_{I\in\mathcal{P}_{m_N}}
\mathbb{E}\left[
Z^\beta(I)
\mathbf{1}_{\{Z^\beta(I)>b/N\}}
\mid\mathcal{F}_{t_N}
\right]
\longrightarrow0.
\end{equation}
\end{proposition}

\begin{proof}
The realized tail has the exact representation
\begin{equation}
\label{eq:realized-tail-integral-representation}
T_N^\beta(b)
=
\int_{\mathbb{T}}
\mathbf{1}_{\{
Z^\beta(Q_{m_N}(\theta))>b/N
\}}
Z^\beta(\mathrm{d}\theta).
\end{equation}
Since \(N/m_N\to1\), by Theorem~\ref{thm:rooted-small-cell-summability} with \(p=1\), almost surely,
\[
NZ^\beta(Q_{m_N}(\theta))
\longrightarrow0,
\qquad
\text{for }Z^\beta\text{-almost every }\theta,
\]
which yields \eqref{eq:realized-moving-cell-tail}.

Since \(\mathcal{P}_{m_N}\) is finite,
\begin{equation}
\label{eq:predictable-tail-conditional-expectation}
B_N^\beta(b)
=
\mathbb{E}\left[
T_N^\beta(b)\mid\mathcal{F}_{t_N}
\right].
\end{equation}
Moreover,
\[
0\leq T_N^\beta(b)\leq Z^\beta(\mathbb{T}),
\]
and the latter variable is integrable.  
Hence, \( S_K := \sup_{N\geq K}T_N^\beta(b) \) decreases to zero both almost surely and in \(L^1\).  
Choose \(K_j\uparrow\infty\) so that \( \mathbb{E}[S_{K_j}]\leq2^{-2j}. \)
For \(N\geq K_j\),
\[
B_N^\beta(b)
\leq
\mathbb{E}\left[
S_{K_j}\mid\mathcal{F}_{t_N}
\right].
\]
Because \((m_N)_{N\ge 1}\) is nondecreasing, the filtration \((\mathcal{F}_{t_N})_{N\ge 1}\) is increasing. 
Doob's weak \(L^1\) inequality gives
\[
\mathbb{P}\left(
\sup_{N\geq K_j}
\mathbb{E}\left[
S_{K_j}\mid\mathcal{F}_{t_N}
\right]
>
2^{-j}
\right)
\leq
2^j\mathbb{E}[S_{K_j}]
\leq
2^{-j}.
\]
The Borel--Cantelli lemma and \eqref{eq:predictable-tail-conditional-expectation} prove \eqref{eq:predictable-moving-cell-tail}.
\end{proof}

\section{Terminal locality and one-scale annular cancellation}
\label{sec:terminal-locality-annular}

We now compare the terminal measure with its conditional expectation at one scale below a prescribed Fourier annulus. 
The compact-range structure gives independence of future fields on sufficiently separated cells. 
The only measure-theoretic point is that the terminal restriction to a cell must first be represented as a measurable function of the coarse field
and the future noise in that cell. 
Once this is done, an alternating two-color decomposition and conditional Bernstein concentration control all frequencies in the annulus simultaneously.

\subsection{Future sigma-algebras and finite-range independence}
\label{subsec:future-sigma-algebras}

For \(t\geq0\), define the future increment field
\[
W_s^{(t)}(\theta)
=
X_{t+s}(\theta)-X_t(\theta),
\qquad
s\geq0,
\quad
\theta\in\mathbb{T}.
\]
Let
\[
\mathcal{F}_t^\circ
=
\sigma\left(
X_v(\theta):
0\leq v\leq t,\ \theta\in\mathbb{T}
\right)
\]
be the raw coarse sigma-algebra, so that \(\mathcal{F}_t\) is its completion. 
For a deterministic set \(A\subset\mathbb{T}\), define
\[
\mathcal{G}_{t,A}^\circ
=
\sigma\left(
W_s^{(t)}(\theta):
s\geq0,\ \theta\in A
\right).
\]

For deterministic subsets \(A,B\subset\mathbb{T}\), write
\[
d_{\mathbb{T}}(A,B)
=
\inf\left\{
d_{\mathbb{T}}(\theta,\theta'):
\theta\in A,\ \theta'\in B
\right\}.
\]

\begin{lemma}
\label{lem:future-finite-range-independence}
Fix \(t\geq0\).
For \(s,s'\geq0\) and \(\theta,\theta'\in\mathbb{T}\),
\begin{equation}
\label{eq:future-increment-covariance}
\mathbb{E}\left[
W_s^{(t)}(\theta)W_{s'}^{(t)}(\theta')
\right]
=
\int_{e^t}^{e^{t+s\wedge s'}}
q_u(\theta-\theta')\,\frac{\mathrm{d}u}{u}.
\end{equation}
The complete future increment field is independent of \(\mathcal{F}_t^\circ\).

Let \(A_1,\ldots,A_m\subset\mathbb{T}\) be deterministic sets satisfying
\[
d_{\mathbb{T}}(A_i,A_j)>e^{-t},
\qquad
i\neq j.
\]
Then, \( \mathcal{F}_t^\circ, \mathcal{G}_{t,A_1}^\circ,\ldots, \mathcal{G}_{t,A_m}^\circ \) are jointly independent sigma-algebras.
\end{lemma}

\begin{proof}
The independent-increment construction of the scale field gives \eqref{eq:future-increment-covariance} and independence from \(\mathcal{F}_t^\circ\).

Fix \(i\neq j\), \(\theta\in A_i\), and \(\theta'\in A_j\). 
For every \(u\geq e^t\),
\[
u\,d_{\mathbb{T}}(\theta,\theta')
>
e^te^{-t}
=
1.
\]
For every integer \(\ell\),
\[
|\theta-\theta'+2\pi\ell|
\geq
d_{\mathbb{T}}(\theta,\theta'),
\]
and the support condition on \(k\) therefore gives
\[
q_u(\theta-\theta')
=
\sum_{\ell\in\mathbb{Z}}
k\bigl(u(\theta-\theta'+2\pi\ell)\bigr)
=
0.
\]
Thus, all cross-covariances between the future Gaussian fields on distinct \(A_i\)'s vanish. 
Every finite subfamily is jointly Gaussian with block-diagonal covariance, hence its blocks are mutually independent.
Passing from cylinder events to the generated sigma-algebras proves the claim.
\end{proof}

\subsection{Terminal locality and two-color independence}
\label{subsec:terminal-conditional-locality}

Let \(I\subset\mathbb{T}\) be a deterministic half-open arc, and let \(\overline{I}\) denote its closure. 
Set
\[
\mathcal{H}_{t,I}^\circ
=
\mathcal{F}_t^\circ
\vee
\mathcal{G}_{t,\overline{I}}^\circ.
\]

\begin{proposition}
\label{prop:finite-terminal-locality}
Fix \(\beta>0\), \(0\leq t<T\), and a deterministic half-open arc \(I\subset\mathbb{T}\). 
The restricted random measure \( \left.Z_T^\beta\right|_I \) has an \(\mathcal{H}_{t,I}^\circ\)-measurable version as a finite positive measure on \(\overline{I}\).
\end{proposition}

\begin{proof}
Put \(u=T-t\). 
For \(\theta\in I\) and \(0\leq v\leq u\),
\[
R_{t+v}^\beta(\theta)
=
R_t^\beta(\theta)
+
\sqrt{2}\,v
-
W_v^{(t)}(\theta).
\]
Consequently,
\begin{equation}
\label{eq:finite-terminal-local-density}
\begin{aligned}
F_T^\beta(\theta)
&=
\left(
R_t^\beta(\theta)
+
\sqrt{2}\,u
-
W_u^{(t)}(\theta)
\right)
\mathbf{1}_{\left\{
\inf_{0\leq s\leq t}
R_s^\beta(\theta)\geq0
\right\}}
\\
&\quad\times
\mathbf{1}_{\left\{
\inf_{0\leq v\leq u}
\left(
R_t^\beta(\theta)
+
\sqrt{2}\,v
-
W_v^{(t)}(\theta)
\right)
\geq0
\right\}}
\\
&\quad\times
\exp\left(
\sqrt{2}X_t(\theta)
+
\sqrt{2}W_u^{(t)}(\theta)
-
t-u
\right).
\end{aligned}
\end{equation}
The quantities \(R_t^\beta(\theta)\), \(X_t(\theta)\), and
\[
\mathbf{1}_{\left\{
\inf_{0\leq s\leq t}R_s^\beta(\theta)\geq0
\right\}}
\]
are \(\mathcal{F}_t^\circ\)-measurable.
By continuity, both infima in
\eqref{eq:finite-terminal-local-density} may be taken over rational
times in their respective intervals.

It follows that, for every \(f\in C(\overline{I})\),
\[
\int_I
f(\theta)F_T^\beta(\theta)\,\mathrm{d}\theta
\]
is \(\mathcal{H}_{t,I}^\circ\)-measurable. 
A countable uniformly dense family in \(C(\overline{I})\) determines the Borel sigma-algebra of the space of finite positive measures on \(\overline{I}\),
which proves the measure-valued assertion.
\end{proof}

The pre-\(t\) barrier event in
\eqref{eq:finite-terminal-local-density} depends on the complete
pre-\(t\) path, which explains the use of
\(\mathcal F_t^\circ\) rather than the sigma-algebra generated by
\(X_t\) alone.

\begin{proposition}
\label{prop:terminal-conditional-locality}
Fix \(\beta>0\), \(t\geq0\), and a deterministic half-open arc \(I\subset\mathbb{T}\). 
There exists a random finite measure \(\widetilde{Z}_{t,I}^\beta\) on \(\overline{I}\) such that

\begin{enumerate}
\item \(\widetilde{Z}_{t,I}^\beta\) is \(\mathcal{H}_{t,I}^\circ\)-measurable;

\item almost surely,
\[
\widetilde{Z}_{t,I}^\beta
=
\left.Z^\beta\right|_I.
\]
\end{enumerate}
\end{proposition}

\begin{proof}
Let \( T_j=t+j, \) for  \( j\geq1, \) and regard \( \mu_j = \left.Z_{T_j}^\beta\right|_I \) as a random finite measure on \(\overline{I}\).  
By Proposition~\ref{prop:finite-terminal-locality}, every \(\mu_j\) is \(\mathcal{H}_{t,I}^\circ\)-measurable.

For every finite \(T_j\), the measure \(Z_{T_j}^\beta\) is absolutely continuous. 
By Corollary~\ref{cor:atomlessness},
\[
Z^\beta(\partial I)=0
\qquad
\text{almost surely}.
\]
Hence, on the event on which \( Z_{T_j}^\beta\xrightarrow{\mathrm{w}}Z^\beta, \) one has
\[
\mu_j\xrightarrow{\mathrm{w}}\left.Z^\beta\right|_I
\]
as measures on \(\overline{I}\). 
Indeed, if \(f\in C(\overline{I})\), extend \(f\) continuously to \(\mathbb{T}\). 
The bounded function \(f\mathbf{1}_I\) is continuous outside \(\partial I\), a \(Z^\beta\)-null set, and therefore
\[
\int_I f\,\mathrm{d}Z_{T_j}^\beta
\longrightarrow
\int_I f\,\mathrm{d}Z^\beta.
\]

The space of finite positive measures on \(\overline{I}\), equipped with the weak topology, is Polish. 
Define
\[
\widetilde{Z}_{t,I}^\beta
=
\begin{cases}
\displaystyle\lim_{j\to\infty}\mu_j,
&\text{if the weak limit exists},\\[1ex]
0,
&\text{otherwise}.
\end{cases}
\]
The set of convergent sequences in a Polish space is Borel, and the limit map is Borel on that set.  
Thus, \(\widetilde{Z}_{t,I}^\beta\) is \(\mathcal{H}_{t,I}^\circ\)-measurable and agrees almost surely with \(\left.Z^\beta\right|_I\).
\end{proof}

Throughout the remainder of this section, \(\beta>0\) denotes a fixed deterministic parameter unless explicitly quantified otherwise.

\begin{proposition}
\label{prop:separated-terminal-conditional-independence}
Fix \(t\geq0\). 
Let \(I_1,\ldots,I_m\) be deterministic half-open arcs such that
\[
d_{\mathbb{T}}(\overline{I_i},\overline{I_j})>e^{-t},
\qquad
i\neq j.
\]
For bounded Borel functionals \(H_j\) on the corresponding spaces of finite measures,
\begin{equation}
\label{eq:terminal-restriction-factorization}
\mathbb{E}\left[
\prod_{j=1}^m
H_j\left(
\left.Z^\beta\right|_{I_j}
\right)
\middle|
\mathcal{F}_t
\right]
=
\prod_{j=1}^m
\mathbb{E}\left[
H_j\left(
\left.Z^\beta\right|_{I_j}
\right)
\middle|
\mathcal{F}_t
\right]
\qquad
\text{almost surely}.
\end{equation}
\end{proposition}

\begin{proof}
By Proposition~\ref{prop:terminal-conditional-locality}, choose versions \(\widetilde{Z}_{t,I_j}^\beta\) measurable with respect to \( \mathcal{F}_t^\circ \vee \mathcal{G}_{t,\overline{I_j}}^\circ. \)
Lemma~\ref{lem:future-finite-range-independence} and the standard conditional factorization yield
\[
\mathbb{E}\left[
\prod_{j=1}^m
H_j\left(
\widetilde{Z}_{t,I_j}^\beta
\right)
\middle|
\mathcal{F}_t^\circ
\right]
=
\prod_{j=1}^m
\mathbb{E}\left[
H_j\left(
\widetilde{Z}_{t,I_j}^\beta
\right)
\middle|
\mathcal{F}_t^\circ
\right].
\]
This follows first for products of indicators and then for bounded Borel functionals by a monotone-class argument.
Passing to the completion of the coarse sigma-algebra does not change the identity.
Since
\[
\widetilde{Z}_{t,I_j}^\beta
=
\left.Z^\beta\right|_{I_j}
\qquad
\text{almost surely},
\]
\eqref{eq:terminal-restriction-factorization} follows.
\end{proof}

We now specialize Proposition~\ref{prop:separated-terminal-conditional-independence} to the dyadic partition \(\mathcal{P}_r\) used in the annular argument.
For \(r\geq1\), define
\[
\mathcal{C}_{r,0}
=
\left\{
I_j^{(r)}:
j\equiv0\pmod2
\right\},
\qquad
\mathcal{C}_{r,1}
=
\left\{
I_j^{(r)}:
j\equiv1\pmod2
\right\}.
\]
If \(I,J\) are distinct cells in the same parity class, then
\[
d_{\mathbb{T}}(\overline{I},\overline{J})
\geq
2\pi2^{-r}.
\]
The number \(2^r\) of cells is even, so the two cells meeting at the circular seam have opposite parity.

We therefore obtain the following corollary.

\begin{corollary}
\label{cor:two-color-fourier-mass-independence}
Fix \(r\geq1\), put \(t=r\log2\), and choose \(\epsilon\in\{0,1\}\).
For \(n\in\mathbb{Z}\) and \(I\in\mathcal{P}_r\), define
\[
V_I(n)
=
\int_I e^{in\theta}\,Z^\beta(\mathrm{d}\theta).
\]
For each \(n\in\mathbb{Z}\), conditionally on \(\mathcal{F}_t\), the random vectors
\[
\left(
V_I(n),Z^\beta(I)
\right),
\qquad
I\in\mathcal{C}_{r,\epsilon},
\]
are mutually independent.
In particular, for every deterministic \(\tau>0\), so are
\[
V_I(n)
\mathbf{1}_{\{Z^\beta(I)\leq\tau\}},
\qquad
I\in\mathcal{C}_{r,\epsilon}.
\]
\end{corollary}

\subsection{Uniform annular approximation}
\label{subsec:uniform-annular-approximation}

We now apply Corollary~\ref{cor:two-color-fourier-mass-independence} at a coarse level just below the Fourier scale.  
For all sufficiently large integers \(N\), define
\[
\Lambda_N
=
\left\{
n\in\mathbb{Z}:
2^N\leq|n|<2^{N+1}
\right\},
\qquad
r_N
=
N-\left\lceil6\log_2N\right\rceil,
\qquad
t_N=r_N\log2.
\]
Choose \(N_0\) so large that \(r_N\ge 1\) for every \(N\ge N_0\) and \((r_N)_{N\ge N_0}\) is nondecreasing.  
Then,
\[
\frac{r_N}{N}\longrightarrow1.
\]

For \(I\in\mathcal{P}_{r_N}\) and \(n\in\mathbb{Z}\), let \(V_I(n)\) be as in Corollary~\ref{cor:two-color-fourier-mass-independence}.
\eqref{eq:positive-truncation-martingale-closure} in Proposition~\ref{prop:periodized-positive-rooted-package} yields
\begin{equation}
\label{eq:cell-conditional-mean}
a_I(n)
:=
\mathbb{E}\left[
V_I(n)\mid\mathcal{F}_{t_N}
\right]
=
\int_Ie^{in\theta}\,
Z_{t_N}^\beta(\mathrm{d}\theta).
\end{equation}
For \(b>0\), define
\[
A_I(n)
=
V_I(n)
\mathbf{1}_{\{Z^\beta(I)\leq b/N\}},
\]
and
\begin{equation}
\label{eq:capped-centered-cell-variable}
Y_I(n)
=
A_I(n)
-
\mathbb{E}\left[
A_I(n)\mid\mathcal{F}_{t_N}
\right].
\end{equation}

\begin{lemma}
\label{lem:annular-capped-cell-bounds}
For every \(N\geq N_0\), \(n\in\mathbb{Z}\), and \(I\in\mathcal{P}_{r_N}\),
\begin{equation}
\label{eq:capped-centered-uniform-bound}
|Y_I(n)|
\leq
\frac{2b}{N}.
\end{equation}
Moreover,
\begin{equation}
\label{eq:capped-centered-second-moment-bound}
\mathbb{E}\left[
|Y_I(n)|^2
\middle|
\mathcal{F}_{t_N}
\right]
\leq
\frac{b}{N}Z_{t_N}^\beta(I).
\end{equation}
Inside either parity class, the variables \(Y_I(n)\) are conditionally independent given \(\mathcal{F}_{t_N}\).
\end{lemma}

\begin{proof}
Since \( |V_I(n)|\leq Z^\beta(I), \) one has
\[
|A_I(n)|\leq\frac{b}{N}.
\]
The same bound holds for the absolute value of its conditional expectation, which gives \eqref{eq:capped-centered-uniform-bound}.
As for \eqref{eq:capped-centered-second-moment-bound},
\begin{equation*}
\mathbb{E}\left[
|Y_I(n)|^2
\middle|
\mathcal{F}_{t_N}
\right]
\leq
\mathbb{E}\left[
Z^\beta(I)^2
\mathbf{1}_{\{Z^\beta(I)\leq b/N\}}
\middle|
\mathcal{F}_{t_N}
\right]
\leq
\frac{b}{N}Z_{t_N}^\beta(I).
\end{equation*}
Conditional independence follows from Corollary~\ref{cor:two-color-fourier-mass-independence}.
\end{proof}

For \(\epsilon\in\{0,1\}\), define
\[
S_{N,\epsilon}(n)
=
\sum_{I\in\mathcal{C}_{r_N,\epsilon}}Y_I(n).
\]
Then,
\[
\sum_{I\in\mathcal{P}_{r_N}}Y_I(n)
=
S_{N,0}(n)+S_{N,1}(n).
\]

\begin{lemma}
\label{lem:centered-annular-concentration}
Fix \(\eta>0\) and \(K\geq1\). 
There exists \(b=b(\eta,K)>0\) such that, with \(Y_I(n)\) defined using this value of \(b\),
\begin{equation}
\label{eq:centered-annular-probability-summability}
\sum_{N=N_0}^{\infty}
\mathbb{P}\biggl(
\max_{n\in\Lambda_N}
\left|
\sum_{I\in\mathcal{P}_{r_N}}Y_I(n)
\right|>\eta,
Z_{t_N}^\beta(\mathbb{T})\leq K
\biggr)
<
\infty.
\end{equation}
Consequently, on
\[
\Omega_{\beta,K}
:=
\left\{
\sup_{N\geq N_0}
Z_{t_N}^\beta(\mathbb{T})
\leq K
\right\},
\]
one has
\begin{equation}
\label{eq:centered-annular-limsup}
\limsup_{N\to\infty}
\max_{n\in\Lambda_N}
\left|
\sum_{I\in\mathcal{P}_{r_N}}Y_I(n)
\right|
\leq
\eta
\qquad
\text{almost surely}.
\end{equation}
\end{lemma}

\begin{proof}
Let \(C,c>0\) be the absolute constants in \cite[Lemma~3.10]{CaiChengFangLiQuXiao2026Rajchman}.
If
\[
\left|
S_{N,0}(n)+S_{N,1}(n)
\right|>\eta,
\]
then \( |S_{N,\epsilon}(n)|>\eta/2 \) for at least one \(\epsilon\in\{0,1\}\).

Fix \(\epsilon\in\{0,1\}\).
By Lemma~\ref{lem:annular-capped-cell-bounds}, we may apply \cite[Lemma~3.10]{CaiChengFangLiQuXiao2026Rajchman} with
\[
R=\frac{2b}{N},
\qquad
V=\frac{b}{N}Z_{t_N}^\beta(\mathbb{T}),
\qquad
t=\frac{\eta}{2},
\]
which yields
\[
\mathbb{P}\left(
|S_{N,\epsilon}(n)|>\frac{\eta}{2}
\middle|
\mathcal{F}_{t_N}
\right)
\leq
C\exp\left(
-\frac{c\eta^2N}
{4b\bigl(Z_{t_N}^\beta(\mathbb{T})+\eta\bigr)}
\right)
\]
almost surely.
Therefore, on \( \{Z_{t_N}^\beta(\mathbb{T})\leq K\}, \)
\begin{equation}
\label{eq:single-frequency-centered-bound}
\mathbf{1}_{\{Z_{t_N}^\beta(\mathbb{T})\leq K\}}
\mathbb{P}\left(
\left|
\sum_{I\in\mathcal{P}_{r_N}}Y_I(n)
\right|>\eta
\middle|
\mathcal{F}_{t_N}
\right)
\leq
2C\exp\left(
-\frac{c\eta^2N}
{4b(K+\eta)}
\right).
\end{equation}

Since \( \#\Lambda_N=2^{N+1}, \) a conditional union bound over \(n\in\Lambda_N\),
followed by taking expectations and using \( \{Z_{t_N}^\beta(\mathbb{T})\leq K\}\in\mathcal{F}_{t_N}, \) gives
\begin{equation}
\label{eq:annular-centered-unconditional-bound}
\mathbb{P}\left(
\max_{n\in\Lambda_N}
\left|
\sum_{I\in\mathcal{P}_{r_N}}Y_I(n)
\right|>\eta,\
Z_{t_N}^\beta(\mathbb{T})\leq K
\right)
\leq
C2^{N+2}
\exp\left(
-\frac{c\eta^2N}
{4b(K+\eta)}
\right).
\end{equation}
Choose \(b=b(\eta,K)>0\) sufficiently small that
\[
\frac{c\eta^2}{4b(K+\eta)}
>
2\log 2.
\]
Then, the right-hand side of \eqref{eq:annular-centered-unconditional-bound} is summable in \(N\), proving \eqref{eq:centered-annular-probability-summability}.
The Borel--Cantelli lemma gives \eqref{eq:centered-annular-limsup} on \(\Omega_{\beta,K}\).
\end{proof}

The event used inside the conditional estimate is \( \{Z_{t_N}^\beta(\mathbb{T})\leq K\} \in\mathcal{F}_{t_N}. \)
The pathwise event \(\Omega_{\beta,K}\) is introduced only after the unconditional summability estimate. 
Moreover,
\[
Z_{t_N}^\beta(\mathbb{T})
\longrightarrow
Z^\beta(\mathbb{T})
\qquad
\text{almost surely},
\]
and hence
\begin{equation}
\label{eq:mass-localization-exhaustion}
\mathbb{P}\left(
\bigcup_{K=1}^{\infty}
\Omega_{\beta,K}
\right)
=
1.
\end{equation}

\begin{lemma}
\label{lem:exact-capped-remainder}
Fix \(b>0\), and define \(Y_I(n)\) by \eqref{eq:capped-centered-cell-variable}.  
Then,
\begin{equation}
\label{eq:uniform-capped-remainder-vanishing}
\sup_{n\in\mathbb{Z}}
\left|
\sum_{I\in\mathcal{P}_{r_N}}
\left(
V_I(n)-a_I(n)-Y_I(n)
\right)
\right|
\longrightarrow0
\end{equation}
almost surely.
\end{lemma}

\begin{proof}
The identity
\[
V_I(n)-a_I(n)-Y_I(n)
=
V_I(n)\mathbf{1}_{\{Z^\beta(I)>b/N\}}
-
\mathbb{E}\left[
V_I(n)\mathbf{1}_{\{Z^\beta(I)>b/N\}}
\middle|
\mathcal{F}_{t_N}
\right]
\]
is immediate from the definitions.
Since \( |V_I(n)|\leq Z^\beta(I), \) summing uniformly in \(n\) gives
\[
\sup_{n\in\mathbb{Z}}
\left|
\sum_{I\in\mathcal{P}_{r_N}}
\left(
V_I(n)-a_I(n)-Y_I(n)
\right)
\right|
\leq
T_N^\beta(b)+B_N^\beta(b),
\]
where \(T_N^\beta(b)\) and \(B_N^\beta(b)\) are the quantities in Proposition~\ref{prop:moving-cell-tails} with \(m_N=r_N\).
By the choice of \(N_0\), the sequence \((r_N)_{N\geq N_0}\) is nondecreasing and satisfies \(r_N/N\to1\).
Proposition~\ref{prop:moving-cell-tails} therefore yields \eqref{eq:uniform-capped-remainder-vanishing}.
\end{proof}

\begin{proposition}\label{prop:one-scale-annular-approximation}
For every fixed \(\beta>0\), almost surely,
\[
\max_{n\in\Lambda_N}
\left|
\widehat{Z^\beta}(n)
-
\widehat{Z_{t_N}^\beta}(n)
\right|
\longrightarrow0.
\]
\end{proposition}

\begin{proof}
The half-open cells form an exact partition of \(\mathbb{T}\), so
\[
\sum_{I\in\mathcal{P}_{r_N}}V_I(n)
=
\widehat{Z^\beta}(n),
\]
and, by \eqref{eq:cell-conditional-mean},
\[
\sum_{I\in\mathcal{P}_{r_N}}a_I(n)
=
\widehat{Z_{t_N}^\beta}(n).
\]
Therefore,
\begin{equation}
\label{eq:annular-approximation-decomposition}
\widehat{Z^\beta}(n)
-
\widehat{Z_{t_N}^\beta}(n)
=
\sum_{I\in\mathcal{P}_{r_N}}Y_I(n)
+
\sum_{I\in\mathcal{P}_{r_N}}
\left(
V_I(n)-a_I(n)-Y_I(n)
\right).
\end{equation}

For every \( (\eta,K)\in\mathbb{Q}_{>0}\times\mathbb{N}, \) choose a deterministic \(b(\eta,K)>0\) for which Lemma~\ref{lem:centered-annular-concentration} holds.
Since \(\mathbb{Q}_{>0}\times\mathbb{N}\) is countable,
intersect the probability-one events supplied by Lemma~\ref{lem:centered-annular-concentration} and Lemma~\ref{lem:exact-capped-remainder} for all such pairs \((\eta,K)\),
together with the probability-one event in \eqref{eq:mass-localization-exhaustion}.

Fix an outcome in this common probability-one event.
Choose \(K\in\mathbb{N}\) such that the outcome belongs to \(\Omega_{\beta,K}\).
For every positive rational \(\eta\), use the corresponding \(b=b(\eta,K)\).  
Then, \eqref{eq:annular-approximation-decomposition}, \eqref{eq:centered-annular-limsup}, and \eqref{eq:uniform-capped-remainder-vanishing} give
\[
\limsup_{N\to\infty}
\max_{n\in\Lambda_N}
\left|
\widehat{Z^\beta}(n)
-
\widehat{Z_{t_N}^\beta}(n)
\right|
\leq
\eta.
\]
Letting \(\eta\downarrow0\) through the positive rationals proves the claim.
\end{proof}

The annuli \(\Lambda_N\) are pairwise disjoint and cover every nonzero integer of sufficiently large absolute value.  
Thus, a later estimate that is uniform on \(\Lambda_N\) will yield convergence along the full sequence \(|n|\to\infty\).

\section{Predictable Fourier decay and the periodized theorem}
\label{sec:periodized-rajchman}

We now complete the Fourier analysis of the periodized compact-range field.
Proposition~\ref{prop:one-scale-annular-approximation} reduces the terminal Fourier coefficient, uniformly on each dyadic annulus,
to that of the coarse measure \(Z_{t_N}^\beta\). 
On a global barrier event this coarse measure has a smooth density. 
Its spatial variation is controlled by an exponential tilt, and periodic integration by parts then gives the required annular decay.

The almost-sure Rajchman property is proved directly for the derivative limit \(M_*'\). 
The Seneta--Heyde theorem is used only afterward, to identify the positive critical measure produced by the periodized star-scale approximation.

\subsection{Predictable annular decay and the Rajchman property}
\label{subsec:predictable-annular-decay}

Recall the global barrier event
\[
E_\beta
=
\left\{
\sup_{t\geq0}
\sup_{\theta\in\mathbb{T}}
\left(
X_t(\theta)-\sqrt{2}\,t
\right)
<\beta
\right\}.
\]
Define, on the whole probability space,
\begin{equation}
\label{eq:predictable-density-definition}
f_t^\beta(\theta)
=
\left(
\beta+\sqrt{2}\,t-X_t(\theta)
\right)
e^{\sqrt{2}X_t(\theta)-t}.
\end{equation}
On \(E_\beta\), the barrier is never crossed, so
\begin{equation}
\label{eq:predictable-density-identification}
Z_t^\beta(\mathrm{d}\theta)
=
f_t^\beta(\theta)\,\mathrm{d}\theta.
\end{equation}
Moreover,
\begin{equation}
\label{eq:predictable-density-derivative}
\partial_\theta f_t^\beta(\theta)
=
\left[
\sqrt{2}
\left(
\beta+\sqrt{2}\,t-X_t(\theta)
\right)
-1
\right]
\partial_\theta X_t(\theta)
e^{\sqrt{2}X_t(\theta)-t}.
\end{equation}

\begin{lemma}
\label{lem:predictable-variation-bound}
For every \(\beta>0\) and \(t\geq0\),
\begin{equation}
\label{eq:predictable-variation-first-moment}
\mathbb{E}
\left\|
\partial_\theta f_t^\beta
\right\|_{L^1(\mathbb{T})}
\leq
C_k e^t
\left(
1+\beta+\sqrt{t}
\right).
\end{equation}
\end{lemma}

\begin{proof}
Recall that
\[
K_t(h)
=
\int_1^{e^t}
q_u(h)\,\frac{\mathrm{d}u}{u}.
\]
At \(h=0\), all nonzero terms in the periodization of \(q_u\) lie outside the support of \(k\). 
Hence, by evenness of \(k\),
\[
q_u'(0)=0,
\qquad
q_u''(0)=u^2k''(0),
\]
and therefore
\[
K_t'(0)=0,
\qquad
K_t''(0)
=
\frac{k''(0)}{2}
\left(
e^{2t}-1
\right).
\]
Stationarity implies
\[
\operatorname{Cov}\left(
X_t(\theta),
\partial_\theta X_t(\theta)
\right)
=
0
\]
and
\[
\operatorname{Var}\left(
\partial_\theta X_t(\theta)
\right)
=
-\frac{k''(0)}{2}
\left(
e^{2t}-1
\right).
\]
Thus, \(X_t(\theta)\) and \(\partial_\theta X_t(\theta)\) are independent Gaussian variables, and
\begin{equation}
\label{eq:spatial-derivative-first-moment}
\mathbb{E}
\left|
\partial_\theta X_t(\theta)
\right|
\leq
C_k e^t.
\end{equation}

By stationarity, independence, and \eqref{eq:predictable-density-derivative},
\begin{equation*}
\mathbb{E}
\left\|
\partial_\theta f_t^\beta
\right\|_{L^1(\mathbb{T})}
=
2\pi
\mathbb{E}
\left|
\partial_\theta X_t(0)
\right| 
\mathbb{E}\left[
\left|
\sqrt{2}
\left(
\beta+\sqrt{2}\,t-X_t(0)
\right)
-1
\right|
e^{\sqrt{2}X_t(0)-t}
\right].
\end{equation*}
Under the exponential tilt \(e^{\sqrt{2}X_t(0)-t}\), the variable \(X_t(0)\) has law \(N(\sqrt{2}\,t,t)\). 
Hence, 
\[
\mathbb{E}\left[
\left|
\sqrt{2}
\left(
\beta+\sqrt{2}\,t-X_t(0)
\right)
-1
\right|
e^{\sqrt{2}X_t(0)-t}
\right]
\le 
C(1+\beta+\sqrt{t}).
\] 
Combining with \eqref{eq:spatial-derivative-first-moment}, we prove \eqref{eq:predictable-variation-first-moment}.
\end{proof}

Recall
\[
\Lambda_N
=
\left\{
n\in\mathbb{Z}:
2^N\leq|n|<2^{N+1}
\right\},
\qquad
r_N
=
N-\left\lceil6\log_2N\right\rceil,
\qquad
t_N=r_N\log2.
\]

\begin{proposition}
\label{prop:predictable-annular-decay}
There is a probability-one event such that, for every integer \(j\geq1\), on \(E_j\),
\begin{equation}
\label{eq:predictable-annular-decay}
\max_{n\in\Lambda_N}
\left|
\widehat{Z_{t_N}^j}(n)
\right|
\leq
N^{-3}
\end{equation}
for all sufficiently large \(N\).
\end{proposition}

\begin{proof}
For each \(j\geq1\),  by Lemma~\ref{lem:predictable-variation-bound} and Markov's inequality,
\[
\mathbb{P}\left(
\left\|
\partial_\theta f_{m\log2}^{\,j}
\right\|_{L^1(\mathbb{T})}
>
2^m m^3
\right)
\leq
C_{j,k}
\frac{1+\sqrt{m}}{m^3}.
\]
The right-hand side is summable in \(m\). 
Hence, by the Borel--Cantelli lemma and a countable intersection over \(j\in\mathbb{N}\), almost surely, for every \(j\geq1\),
\begin{equation}
\label{eq:integer-time-variation-bound}
\left\|
\partial_\theta f_{m\log2}^{\,j}
\right\|_{L^1(\mathbb{T})}
\leq
2^m m^3
\end{equation}
for all sufficiently large \(m\).

Let \(\Omega_0\) denote the probability-one event constructed above.
Fix \(j\geq1\). 
On \(\Omega_0\cap E_j\), \eqref{eq:predictable-density-identification} and periodic integration by parts imply, for \(n\neq0\),
\begin{equation}
\label{eq:predictable-fourier-variation-bound}
\left|
\widehat{Z_t^j}(n)
\right|
\leq
\frac{1}{|n|}
\left\|
\partial_\theta f_t^j
\right\|_{L^1(\mathbb{T})}.
\end{equation}
Taking \(t=t_N=r_N\log2\), estimate \eqref{eq:integer-time-variation-bound} and \eqref{eq:predictable-fourier-variation-bound} yield,
for all sufficiently large \(N\) and every \(n\in\Lambda_N\),
\[
\left|
\widehat{Z_{t_N}^j}(n)
\right|
\leq
2^{r_N-N}r_N^3,
\]
which proves \eqref{eq:predictable-annular-decay} by the definition of \(r_N\).
\end{proof}

We now combine the predictable estimate with the one-scale annular approximation.

\begin{theorem}[The periodized derivative measure is Rajchman]
\label{thm:periodized-derivative-rajchman}
Almost surely,
\begin{equation}
\label{eq:periodized-derivative-Rajchman}
\lim_{|n|\to\infty}
\widehat{M_*'}(n)
=
0.
\end{equation}
\end{theorem}

\begin{proof}
By a countable intersection, we may work on a probability-one event on which \(S<\infty\)
and the conclusions of Propositions~\ref{prop:one-scale-annular-approximation}
and \ref{prop:predictable-annular-decay} hold simultaneously for every \(j\in\mathbb{N}\).  
Fix an outcome in this event.
Choose an integer \(j>S\).  
Then, \(E_j\) occurs and \(M_*'=Z^j\) by \eqref{eq:barrier-identification}. 
Hence,
\[
\max_{n\in\Lambda_N}
\left|
\widehat{M_*'}(n)
\right|
\leq
\max_{n\in\Lambda_N}
\left|
\widehat{Z^j}(n)-\widehat{Z_{t_N}^j}(n)
\right|
+
\max_{n\in\Lambda_N}
\left|
\widehat{Z_{t_N}^j}(n)
\right|
\longrightarrow0.
\]
Since the annuli \(\Lambda_N\) cover all integers of sufficiently large absolute value, \eqref{eq:periodized-derivative-Rajchman} follows.
\end{proof}

The integer \(j\) is selected only after all deterministic integer-barrier assertions have been made simultaneous. 
In particular, no theorem proved for fixed deterministic \(\beta\) is applied with a random barrier parameter.

\subsection{Periodized Seneta--Heyde convergence}
\label{subsec:periodized-seneta-heyde}

We next identify the positive critical measure associated with the periodized scale approximation. 
The local Seneta--Heyde theorem is stated for the real-line compact-range field. 
Exact local lifting transfers it to short circle arcs, and a disjoint partition then gives a global measure-level statement.

\begin{proposition}\label{prop:periodized-seneta-heyde}
As \(t\to\infty\),
\begin{equation}
\label{eq:periodized-seneta-heyde}
\sqrt{t}\,
M_t^{\sqrt{2}}
\longrightarrow
\sqrt{\frac{2}{\pi}}\,
M_*'
\end{equation}
in probability in the weak topology on finite positive measures on \(\mathbb{T}\).
\end{proposition}

\begin{proof}
Recall the dyadic partition \(\mathcal{P}_m\) from Section~\ref{sec:rooted-small-cells}, whose cells have length \( \ell_m=2\pi2^{-m}. \)
Choose an integer \(m_0\) such that
\[
\ell_{m_0}<\frac{1}{2}.
\]
Fix \(m\geq m_0\) and a cell \(I\in\mathcal{P}_m\).  
Its interior admits a lift to a real interval of diameter less than \(1/2\).
By Lemma~\ref{lem:periodized-positive-definite-local-lifting},
the complete space-time field on this lift has exactly the compact-range law used in \cite{DRSVRenormalization}. 
The local Seneta--Heyde theorem \cite[Theorem~5]{DRSVRenormalization} therefore gives
\begin{equation}
\label{eq:local-cell-seneta-heyde}
\sqrt{t}\,
M_t^{\sqrt{2}}(I^\circ)
\longrightarrow
\sqrt{\frac{2}{\pi}}\,
M_{\mathrm{loc}}'(I^\circ)
\end{equation}
in probability, where \(M_{\mathrm{loc}}'\) denotes the local derivative limit.

We next identify \(M_{\mathrm{loc}}'(I^\circ)\) with \(M_*'(I)\).
By a countable intersection, we may work on a probability-one event on which the relevant conclusions of \cite[Theorem~17]{DRSVDerivative} hold,
\(S<\infty\), \(M_t^{\sqrt{2}}(\mathbb{T})\to0\), and, for every \(j\in\mathbb{N}\),
\[
Z_t^j\xrightarrow{\mathrm{w}}Z^j,
\qquad
Z^j\ \text{is atomless}.
\]
Fix an outcome in this event and choose an integer \(j>S\). 
Since \(I\) is then a \(Z^j\)-continuity set,   \eqref{eq:barrier-identification} and \eqref{eq:derivative-truncation-identity} imply, on \(E_j\),
\[
M_t'(I)
=
Z_t^j(I)-jM_t^{\sqrt{2}}(I)
\longrightarrow
Z^j(I)
=
M_*'(I)
\qquad
\text{almost surely}.
\]
On the other hand, exact local lifting and \cite[Theorem~17]{DRSVDerivative} imply
\[
M_t'(I^\circ)
\longrightarrow
M_{\mathrm{loc}}'(I^\circ)
\qquad
\text{almost surely}.
\]
Since \(M_t'\) is absolutely continuous at every finite time, \(M_t'(I^\circ)=M_t'(I)\), and hence
\[
M_{\mathrm{loc}}'(I^\circ)=M_*'(I)
\qquad
\text{almost surely}.
\]
Together with the absolute continuity of \(M_t^{\sqrt{2}}\), \eqref{eq:local-cell-seneta-heyde} therefore becomes
\begin{equation}
\label{eq:half-open-cell-seneta-heyde}
\sqrt{t}\,
M_t^{\sqrt{2}}(I)
\longrightarrow
\sqrt{\frac{2}{\pi}}\,
M_*'(I)
\end{equation}
in probability. 
Summing over the finitely many cells of \(\mathcal{P}_m\) gives
\begin{equation}
\label{eq:total-mass-seneta-heyde}
\sqrt{t}\,
M_t^{\sqrt{2}}(\mathbb{T})
\longrightarrow
\sqrt{\frac{2}{\pi}}\,
M_*'(\mathbb{T})
\end{equation}
in probability.

Let \(f\in C(\mathbb{T})\). 
For each \(I\in\mathcal{P}_m\), choose \(\theta_I\in I\), and define the cellwise constant function
\[
f_m(\theta)
=
\sum_{I\in\mathcal{P}_m}
f(\theta_I)\mathbf{1}_I(\theta).
\]
If
\[
\omega_f(\delta)
:=
\sup\left\{
|f(\theta)-f(\theta')|:
d_{\mathbb{T}}(\theta,\theta')\leq\delta
\right\},
\]
then
\begin{equation}
\label{eq:cellwise-constant-approximation}
\|f-f_m\|_{L^\infty(\mathbb{T})}
\leq
\omega_f(\ell_m)
\longrightarrow0.
\end{equation}
For fixed \(m\), \eqref{eq:half-open-cell-seneta-heyde} yields
\begin{equation}
\label{eq:cellwise-test-function-seneta-heyde}
\int_{\mathbb{T}}
f_m(\theta)\sqrt{t}\,
M_t^{\sqrt{2}}(\mathrm{d}\theta)
\longrightarrow
\sqrt{\frac{2}{\pi}}
\int_{\mathbb{T}}
f_m(\theta)\,M_*'(\mathrm{d}\theta)
\end{equation}
in probability.  
By \eqref{eq:total-mass-seneta-heyde}, the family \( \sqrt{t}\,M_t^{\sqrt{2}}(\mathbb{T}) \) is tight as \(t\to\infty\).
Together with \eqref{eq:cellwise-constant-approximation}, \eqref{eq:cellwise-test-function-seneta-heyde}, the triangle inequality,
and choosing \(m\) large before letting \(t\to\infty\), we obtain
\[
\int_{\mathbb{T}}
f(\theta)\sqrt{t}\,
M_t^{\sqrt{2}}(\mathrm{d}\theta)
\longrightarrow
\sqrt{\frac{2}{\pi}}
\int_{\mathbb{T}}
f(\theta)\,M_*'(\mathrm{d}\theta)
\]
in probability.
Applying this conclusion to a countable uniformly dense determining family in \(C(\mathbb{T})\) proves \eqref{eq:periodized-seneta-heyde}.
\end{proof}

The disjoint partition is used to assemble the local statements into a global measure identity.  
An overlapping finite cover alone would not give the total-mass convergence required in the final approximation step.

Combining Proposition~\ref{prop:periodized-seneta-heyde} with Theorem~\ref{thm:periodized-derivative-rajchman}, we obtain the following corollary.

\begin{corollary}
\label{cor:periodized-scale-chaos-rajchman}
Denote by \(M_{\mathrm{star}}^{\mathrm{crit}}\) the limit of the normalized critical measures \( \sqrt{t}\,M_t^{\sqrt{2}} \) in the periodized star-scale approximation.
Then, almost surely,
\begin{equation}
\label{eq:periodized-scale-critical-measure}
M_{\mathrm{star}}^{\mathrm{crit}}
=
\sqrt{\frac{2}{\pi}}\,
M_*'.
\end{equation}
In particular, \(M_{\mathrm{star}}^{\mathrm{crit}}\) is almost surely a Rajchman measure.
\end{corollary}

The identification of \(M_{\mathrm{star}}^{\mathrm{crit}}\) with the convolution-mollifier critical chaos associated with the limiting generalized field will be established in Section~\ref{sec:exact-circle-transfer}.

\section{Transfer to the exact circle field}
\label{sec:exact-circle-transfer}

We now transfer the periodized theorem to the exact circle field. 
The first step is an explicit Fourier comparison of the two limiting covariance kernels. 
The missing covariance is positive definite and smooth, and can therefore be supplied by an independent smooth Gaussian field.  
Equality in law of the resulting generalized fields is not, however, sufficient to identify their critical chaos measures.  
We accordingly compare the star-scale and convolution approximations themselves and apply the two critical uniqueness theorems of Junnila
and Saksman \cite{JunnilaSaksman2017} in their distinct roles:
Theorem~1.1 gives convergence in distribution, while Theorem~4.4 upgrades this to convergence in probability on the original coupling.
A final constant-mode identity then transfers the Rajchman property to the canonical critical chaos of the exact circle field.

\subsection{Fourier covariance correction}
\label{subsec:fourier-covariance-correction}

We use the Fourier conventions
\[
\widehat{k}(\xi)
=
\int_{\mathbb{R}}
k(x)e^{-i\xi x}\,\mathrm{d}x
\qquad
\text{and}
\qquad
[F]_n
:=
\frac{1}{2\pi}
\int_0^{2\pi}
F(\theta)e^{-in\theta}\,\mathrm{d}\theta.
\]
Recall
\[
q_u(\theta)
=
\sum_{\ell\in\mathbb{Z}}
k\bigl(u(\theta+2\pi\ell)\bigr),
\qquad
K_t(\theta)
=
\int_1^{e^t}
q_u(\theta)\,\frac{\mathrm{d}u}{u}.
\]
By \eqref{eq:periodized-kernel-fourier-coefficient} and the evenness of \(\widehat{k}\), for \(n\neq0\),
\[
[K_t]_n
=
\frac{1}{2\pi}
\int_1^{e^t}
\widehat{k}\left(\frac{n}{u}\right)
\frac{\mathrm{d}u}{u^2}=
\frac{1}{2\pi|n|}
\int_{|n|e^{-t}}^{|n|}
\widehat{k}(v)\,\mathrm{d}v.
\]
For \(n=0\),
\[
[K_t]_0
=
a_0(1-e^{-t}),
\qquad
a_0
:=
\frac{\widehat{k}(0)}{2\pi}.
\]
Thus, as \(t\to\infty\), the Fourier coefficients of \(K_t\) converge to those of the distribution \(K_*\) characterized by
\[
[K_*]_0=a_0,
\qquad
[K_*]_n
=
a_n
:=
\frac{1}{2\pi|n|}
\int_0^{|n|}
\widehat{k}(v)\,\mathrm{d}v,
\qquad
n\neq0.
\]
For \(h\neq0\), the compact support of \(k\) gives \(q_u(h)=0\) whenever \(u>d_{\mathbb{T}}(h,0)^{-1}\).
Hence, off the diagonal, the distribution \(K_*\) is represented by
\[
K_*(h)
=
\int_1^\infty q_u(h)\,\frac{\mathrm{d}u}{u},
\qquad
h\neq0.
\]

As a distribution on \(\mathbb{T}\), the exact circle covariance is
\begin{equation}
\label{eq:exact-circle-covariance-Fourier-series}
K_{\mathrm{cir}}(\theta)
:=
\log\frac{1}{|e^{i\theta}-1|}
=
\sum_{n\neq0}
\frac{e^{in\theta}}{2|n|}.
\end{equation}
Fourier inversion, evenness, and \(k(0)=1\) imply
\begin{equation}
\label{eq:hat-k-half-line-integral}
\int_0^\infty
\widehat{k}(v)\,\mathrm{d}v
=
\pi.
\end{equation}
For \(n\neq0\), set
\[
b_n
:=
\frac{1}{2|n|}-a_n
=
\frac{1}{2\pi|n|}
\int_{|n|}^{\infty}
\widehat{k}(v)\,\mathrm{d}v.
\]

\begin{lemma}
\label{lem:smooth-covariance-correction}
The coefficients \(b_n\) are nonnegative and satisfy
\[
|n|^A b_n\longrightarrow0
\qquad
\text{as }|n|\to\infty
\]
for every \(A>0\). 
Consequently,
\begin{equation}
\label{eq:smooth-correction-covariance}
B_{\mathrm{sm}}(\theta)
:=
\sum_{n\neq0}
b_ne^{in\theta}
\end{equation}
defines a real-valued even \(C^\infty\) positive-definite function on \(\mathbb{T}\), and
\begin{equation}
\label{eq:covariance-correction-identity}
B_{\mathrm{sm}}
=
K_{\mathrm{cir}}+a_0-K_*
\end{equation}
as distributions.
\end{lemma}

\begin{proof}
Since \(\widehat{k}\geq0\), one has \(b_n\geq0\). 
Moreover, \(k\in C_c^\infty(\mathbb{R})\), so \(\widehat{k}\) is Schwartz.  
Thus, for every \(M\geq1\),
\[
0\leq b_n
\leq
C_M|n|^{-M-1}.
\]
The rapid decay and the symmetry \(b_{-n}=b_n\in\mathbb{R}\) imply that \(B_{\mathrm{sm}}\) is real-valued, even, and \(C^\infty\). 
Since all of its Fourier coefficients are nonnegative, it is positive definite.
Finally, the zero Fourier coefficient of \(K_{\mathrm{cir}}+a_0-K_*\) is zero, while for \(n\neq0\),
\[
[K_{\mathrm{cir}}+a_0-K_*]_n
=
\frac{1}{2|n|}-a_n
=
b_n.
\]
This proves \eqref{eq:covariance-correction-identity}.
\end{proof}

We next construct the limiting star field.

\begin{lemma}
\label{lem:scale-field-negative-Sobolev-limit}
For every \(s>0\), the process \((X_t)_{t\geq0}\) is an \(L^2\)-bounded \(H^{-s}(\mathbb{T})\)-valued martingale. 
There exists a centered generalized Gaussian field \(X_*\) such that
\begin{equation}
\label{eq:scale-field-negative-Sobolev-limit}
X_t
\longrightarrow
X_*
\qquad
\text{almost surely and in }L^2
\text{ in }H^{-s}(\mathbb{T}).
\end{equation}
The covariance of \(X_*\) is \(K_*\).
\end{lemma}

\begin{proof}
Let
\[
\widehat{X_t}(n)
=
\frac{1}{2\pi}
\int_0^{2\pi}
X_t(\theta)e^{-in\theta}\,\mathrm{d}\theta.
\]
Stationarity implies
\[
\mathbb{E}\left|
\widehat{X_t}(n)
\right|^2
=
[K_t]_n.
\]
Combining with \eqref{eq:hat-k-half-line-integral}, we obtain that for \(n\neq0\), 
\[
0\leq [K_t]_n
\leq
a_n
\leq
\frac{1}{2|n|},
\]
while \(0\leq[K_t]_0\leq a_0\). 
Hence,
\[
\sup_{t\geq0}
\mathbb{E}
\|X_t\|_{H^{-s}(\mathbb{T})}^2
\leq
a_0
+
\sum_{n\neq0}
\frac{1}{(1+n^2)^s}
\frac{1}{2|n|}
<
\infty.
\]
The independent centered scale increments make \((X_t)\) an \(H^{-s}(\mathbb{T})\)-valued martingale,
so the Hilbert-space martingale convergence theorem proves \eqref{eq:scale-field-negative-Sobolev-limit}. 
Passing to the limit in the covariance coefficients identifies the covariance of \(X_*\) with \(K_*\).
\end{proof}

Let \((\xi_n)_{n\geq1}\) and \((\eta_n)_{n\geq1}\) be independent sequences of standard real Gaussian variables, independent of the scale field, and define
\begin{equation}
\label{eq:smooth-Gaussian-correction}
Y(\theta)
=
\sum_{n=1}^{\infty}
\sqrt{2b_n}
\left(
\xi_n\cos(n\theta)
+
\eta_n\sin(n\theta)
\right).
\end{equation}
By Lemma~\ref{lem:smooth-covariance-correction},  the coefficients \(b_n\) decay faster than any power.  
Hence, \(Y\) has an almost surely \(C^\infty\) version, with covariance
\[
\mathbb{E}[Y(\theta)Y(\theta')]
=
B_{\mathrm{sm}}(\theta-\theta').
\]
Set
\begin{equation}
\label{eq:smooth-correction-variance}
\sigma_Y^2
=
\mathbb{E}[Y(\theta)^2]
=
B_{\mathrm{sm}}(0).
\end{equation}
Let \(G_0\) be a centered Gaussian variable with
\[
\mathbb{E}[G_0^2]=a_0,
\]
independent of the exact circle field \(\phi\).

\begin{proposition}
\label{prop:corrected-field-law}
As centered generalized Gaussian fields on \(\mathbb{T}\),
\begin{equation}
\label{eq:corrected-field-law}
X_*+Y
\stackrel{\mathrm{law}}{=}
\phi+G_0.
\end{equation}
\end{proposition}

\begin{proof}
By \eqref{eq:covariance-correction-identity},
\[
\operatorname{Cov}(X_*+Y)
=
K_*+B_{\mathrm{sm}}
=
K_{\mathrm{cir}}+a_0,
\]
which is also the covariance of \(\phi+G_0\).  
Since centered generalized Gaussian fields are determined in law by their covariance forms, \eqref{eq:corrected-field-law} follows.
\end{proof}

The correction \(Y\) has no constant Fourier mode,  while the zero-mode variance \(a_0\) of the limiting star field is represented on the exact-circle side by \(G_0\).

\subsection{Rajchman multipliers and the corrected scale limit}
\label{subsec:rajchman-multipliers-corrected-limit}

We use the following elementary stability property.

\begin{lemma}
\label{lem:rajchman-multiplier}
Let \(\mu\) be a finite Rajchman measure on \(\mathbb{T}\), and let \( g\in L^1(\mu). \)
Then, \(g\mu\) is Rajchman.  
If \(g\) is continuous and strictly positive on \(\mathbb{T}\), then
\[
\mu\text{ is Rajchman}
\quad\Longleftrightarrow\quad
g\mu\text{ is Rajchman}.
\]
\end{lemma}

\begin{proof}
Approximate \(g\) in \(L^1(\mu)\) by a trigonometric polynomial
\[
p(\theta)
=
\sum_{|j|\leq J}
c_je^{ij\theta}.
\]
Then,
\[
\widehat{p\mu}(n)
=
\sum_{|j|\leq J}
c_j\widehat{\mu}(n+j)
\longrightarrow0.
\]
The Fourier transform of \((g-p)\mu\) is uniformly bounded by \(\|g-p\|_{L^1(\mu)}\).  
Letting the approximation error tend to zero proves the first assertion.  

For the converse, suppose that \(g\) is continuous and strictly positive and that \(g\mu\) is Rajchman.
Since \(1/g\) is bounded and continuous, applying the first assertion to \(g\mu\) with multiplier \(1/g\) shows that \(\mu\) is Rajchman.
\end{proof}

Put \( H=X_*+Y,~ H_t=X_t+Y. \)
Since \(X_t\) and \(Y\) are independent,
\[
\mathbb{E}[H_t(\theta)^2]
=
t+\sigma_Y^2.
\]
Define the positive critical scale measures
\begin{equation}
\label{eq:corrected-scale-critical-approximations}
\mu_t^H(\mathrm{d}\theta)
=
\sqrt{t}
\exp\left(
\sqrt{2}H_t(\theta)
-
\mathbb{E}[H_t(\theta)^2]
\right)
\mathrm{d}\theta.
\end{equation}
Thus,
\begin{equation}
\label{eq:corrected-scale-factorization}
\mu_t^H
=
e^{\sqrt{2}Y-\sigma_Y^2}
\sqrt{t}\,M_t^{\sqrt{2}}.
\end{equation}

\begin{proposition}
\label{prop:corrected-scale-critical-limit}
As \(t\to\infty\),
\begin{equation}
\label{eq:corrected-scale-critical-limit}
\mu_t^H
\longrightarrow
\nu_H
\end{equation}
in probability in the weak topology, where
\begin{equation}
\label{eq:corrected-scale-limit-measure}
\nu_H
:=
e^{\sqrt{2}Y-\sigma_Y^2}
\sqrt{\frac{2}{\pi}}\,
M_*'.
\end{equation}
The random measure \(\nu_H\) is almost surely finite, atomless, and Rajchman.
\end{proposition}

\begin{proof}
By \eqref{eq:corrected-scale-factorization} and Proposition~\ref{prop:periodized-seneta-heyde},
\eqref{eq:corrected-scale-critical-limit} follows from the continuity of multiplication by the almost surely continuous function \(e^{\sqrt{2}Y-\sigma_Y^2}\).
Finiteness and atomlessness are inherited from \(M_*'\),
while Theorem~\ref{thm:periodized-derivative-rajchman} and Lemma~\ref{lem:rajchman-multiplier} imply that \(\nu_H\) is Rajchman.
\end{proof}

\subsection{Scale and mollifier covariance comparison}
\label{subsec:scale-mollifier-comparison}

Let \( \rho\in C_c^\infty(\mathbb{R}) \) be nonnegative and satisfy
\[
\int_{\mathbb{R}}\rho(x)\,\mathrm{d}x=1.
\]
Define its periodized rescaling by
\begin{equation}
\label{eq:periodized-mollifier}
\rho_\varepsilon(\theta)
=
\sum_{\ell\in\mathbb{Z}}
\frac{1}{\varepsilon}
\rho\left(
\frac{\theta+2\pi\ell}{\varepsilon}
\right).
\end{equation}
All convolutions below are periodic.  
Define
\[
\psi(x)
=
\int_{\mathbb{R}}
\rho(y)\rho(y+x)\,dy.
\]
Then, \(\psi\) is nonnegative, even, smooth, compactly supported, and has integral one.
We use the same notation \(\psi_\varepsilon\) for its periodized rescaling.

For \(\varepsilon=e^{-t}\), the stationary covariance kernels corresponding respectively to \(H_t\) and \(\rho_\varepsilon*H\) are
\[
C_t^A(h)
:=
K_t(h)+B_{\mathrm{sm}}(h)
\qquad
\text{and}
\qquad
C_t^B(h)
:=
\psi_\varepsilon*
\left(
K_{\mathrm{cir}}+a_0
\right)(h).
\]
The common limiting covariance is
\[
C(h)
:=
K_{\mathrm{cir}}(h)+a_0
=
K_*(h)+B_{\mathrm{sm}}(h).
\]

For \(0<\varepsilon<1\) and \(0\leq r\leq\pi\), define
\begin{equation}
\label{eq:truncated-logarithm}
L_\varepsilon(r)
=
\log^+\left(
\frac{1}{r\vee\varepsilon}
\right).
\end{equation}
There is a bounded continuous function \(g\) on \(\mathbb{T}\) such that, for \(h\neq0\),
\begin{equation}
\label{eq:circle-logarithm-plus-bounded-part}
C(h)
=
\log^+\left(
\frac{1}{d_{\mathbb{T}}(h,0)}
\right)
+
g(h).
\end{equation}

\begin{lemma}[Uniform scale-versus-mollifier covariance comparison]
\label{lem:scale-mollifier-covariance-comparison}
There exist \(t_0<\infty\) and \(C_0<\infty\), depending only on \(k\) and \(\rho\), such that
\begin{equation}
\label{eq:uniform-scale-mollifier-comparison}
\sup_{t\geq t_0}
\sup_{h\in\mathbb{T}}
\left|
C_t^A(h)-C_t^B(h)
\right|
\leq C_0.
\end{equation}
Moreover, for every \(\delta>0\),
\begin{equation}
\label{eq:off-diagonal-scale-mollifier-convergence}
\sup_{\substack{h\in\mathbb{T}\\
d_{\mathbb{T}}(h,0)\geq\delta}}
\left|
C_t^A(h)-C_t^B(h)
\right|
\longrightarrow0
\qquad
\text{as }t\to\infty.
\end{equation}
\end{lemma}

\begin{proof}
Put \( \varepsilon=e^{-t} \) and \( r=d_{\mathbb{T}}(h,0). \)

\medskip

\noindent\emph{The scale covariance.}
If \(r>\varepsilon\), then for every \(u\geq e^t\) and every integer \(\ell\),
\[
u|h+2\pi\ell|
\geq
ur
>
1.
\]
Hence, \(q_u(h)=0\) for every \(u\geq e^t\), and therefore
\begin{equation}
\label{eq:scale-covariance-exact-off-range}
K_t(h)=K_*(h),
\qquad
C_t^A(h)=C(h).
\end{equation}

Suppose that \(0<r\leq\varepsilon\). 
For \(1\leq u\leq\varepsilon^{-1}\), only the principal periodization term contributes and \( q_u(h)=k(ur). \)
Thus,
\[
K_t(h)-t
=
\int_r^{r/\varepsilon}
\frac{k(v)-1}{v}\,\mathrm{d}v.
\]
Because \(k\) is even and smooth with \(k(0)=1\),
\[
K_t(h)=t+O_k(1),
\qquad
0<r\leq\varepsilon.
\]
At \(r=0\), one has exactly \(K_t(0)=t\).

If \(r\leq\varepsilon\), then \( L_\varepsilon(r) = \log(1/\varepsilon) = t, \) and the boundedness of \(B_{\mathrm{sm}}\) gives
\[
C_t^A(h)
=
L_\varepsilon(r)+O_k(1).
\]
If \(r>\varepsilon\), then \eqref{eq:scale-covariance-exact-off-range} and \eqref{eq:circle-logarithm-plus-bounded-part} give
\[
C_t^A(h)
=L_\varepsilon(r)+g(h).
\]
Since \(g\) is bounded, the two regions together yield, after increasing \(t_0\) if necessary,
\begin{equation}
\label{eq:scale-truncated-logarithm-comparison}
\sup_{t\geq t_0}
\sup_{h\in\mathbb{T}}
\left|
C_t^A(h)
-
L_{e^{-t}}\left(
d_{\mathbb{T}}(h,0)
\right)
\right|
\leq C.
\end{equation}

\medskip

\noindent\emph{The mollified covariance.}
Choose \(R\geq1\) such that
\[
\operatorname{supp}\psi\subset[-R,R].
\]
Write
\[
\lambda(h)
=
\log^+\frac{1}{d_{\mathbb{T}}(h,0)},
\qquad h\neq0,
\]
viewed as a locally integrable function on \(\mathbb{T}\).
Since \(\psi_\varepsilon\) is nonnegative and has integral one,
\[
\left|
(\psi_\varepsilon*g)(h)
\right|
\leq
\|g\|_{L^\infty(\mathbb{T})}.
\]
It therefore remains to compare \( \psi_\varepsilon*\lambda \) with \(L_\varepsilon(r)\).

Take \(t_0\) sufficiently large that, for \(t\geq t_0\),
\[
3R\varepsilon<1.
\]
Suppose first that \( r\leq2R\varepsilon. \)
Using a local lift, write \( h=\varepsilon y,~ |y|\leq2R. \)
For \(z\in\operatorname{supp}\psi\), one has
\[
\lambda(h-\varepsilon z)
=
\log\frac{1}{|\varepsilon y-\varepsilon z|}
\]
for almost every \(z\), and therefore
\[
(\psi_\varepsilon*\lambda)(h)
=
\log\frac{1}{\varepsilon}
+
\int_{\mathbb{R}}
\psi(z)
\log\frac{1}{|y-z|}
\,\mathrm{d}z.
\]
Since the last integral is bounded uniformly for \(|y|\leq2R\), we obtain
\[
C_t^B(h)
=
\log\frac{1}{\varepsilon}
+
O_{k,\rho}(1),
\qquad
r\leq2R\varepsilon.
\]
If \(r\leq\varepsilon\), then \( L_\varepsilon(r)=\log(1/\varepsilon) \).
If instead \( \varepsilon<r\leq2R\varepsilon \), then
\[
\left|
\log\frac{1}{\varepsilon}
-
L_\varepsilon(r)
\right|
\leq
\log(2R).
\]
Thus,
\[
C_t^B(h)
=
L_\varepsilon(r)+O_{k,\rho}(1),
\qquad
r\leq2R\varepsilon.
\]

It remains to consider \(r>2R\varepsilon\).
Fix a deterministic \(r_0\in(0,2/3)\), and increase \(t_0\) further so that for \(t\geq t_0\),
\[
R\varepsilon<\frac{r_0}{2}.
\]
First, suppose that \( 2R\varepsilon<r<r_0. \)
For every \(z\in\operatorname{supp}\psi\), 
\[
\frac{r}{2}
\leq
d_{\mathbb{T}}(h-\varepsilon z,0)
\leq
\frac{3r}{2},
\]
which implies that
\[
(\psi_\varepsilon*\lambda)(h)
=
\log\frac{1}{r}
+
O(1).
\]
Here \(\varepsilon<r<1\), so \( L_\varepsilon(r) = \log(1/r). \)
Together with the bounded \(g\)-term, this yields
\[
C_t^B(h)
=
L_\varepsilon(r)+O_{k}(1),
\qquad
2R\varepsilon<r<r_0.
\]

Finally, suppose that \(r\geq r_0\).
For \(z\in\operatorname{supp}\psi\),
\[
d_{\mathbb{T}}(h-\varepsilon z,0)
\geq
\frac{r_0}{2}.
\]
Hence, \(C_t^B(h)\) is uniformly bounded in this region. 
Moreover,
\[
0\leq L_\varepsilon(r)
\leq
\log^+\frac{1}{r_0},
\]
so, for \(r\geq r_0\),
\[
C_t^B(h)
=
L_\varepsilon(r)+O_{k}(1).
\]

Combining the three regions gives
\begin{equation}
\label{eq:mollifier-truncated-logarithm-comparison}
\sup_{t\geq t_0}
\sup_{h\in\mathbb{T}}
\left|
C_t^B(h)
-
L_{e^{-t}}\left(
d_{\mathbb{T}}(h,0)
\right)
\right|
\leq C.
\end{equation}
Equations \eqref{eq:scale-truncated-logarithm-comparison} and \eqref{eq:mollifier-truncated-logarithm-comparison} prove \eqref{eq:uniform-scale-mollifier-comparison}.

\medskip

\noindent\emph{Convergence away from the diagonal.}
Fix \(\delta>0\).
If \(t\) is sufficiently large that \( e^{-t}<\delta, \) then \eqref{eq:scale-covariance-exact-off-range} gives
\[
C_t^A(h)=C(h)
\qquad
\text{whenever }
d_{\mathbb{T}}(h,0)\geq\delta.
\]
If also \( Re^{-t}<\delta/2, \) then for every \(z\in\operatorname{supp}\psi\) and every such \(h\),
\[
d_{\mathbb{T}}(h-e^{-t}z,0)
\geq
\frac{\delta}{2}.
\]
Thus, the function \(C\) is uniformly continuous there.  
Hence, the approximate-identity property gives
\[
\sup_{\substack{h\in\mathbb{T}\\
d_{\mathbb{T}}(h,0)\geq\delta}}
|C_t^B(h)-C(h)|
\longrightarrow0,
\]
which proves \eqref{eq:off-diagonal-scale-mollifier-convergence}.
\end{proof}

The zero Fourier coefficients are
\begin{equation}
\label{eq:two-approximation-zero-modes}
[C_t^A]_0
=
a_0(1-e^{-t}),
\qquad
[C_t^B]_0=a_0.
\end{equation}
Thus, the zero-mode discrepancy is exactly \(-a_0e^{-t}\).  
It is already included in Lemma~\ref{lem:scale-mollifier-covariance-comparison};  no additional renormalization is needed.

\subsection{Critical uniqueness}
\label{subsec:critical-uniqueness}

Let \( \varepsilon_j\downarrow0 \) be an arbitrary strictly decreasing deterministic sequence, and put \( t_j = \log(1/\varepsilon_j). \)
After deleting finitely many terms, assume \(t_j\geq t_0\). 
Define
\[
U_j
=
\sqrt{2}\,H_{t_j}
\qquad
\text{and}
\qquad
V_j
=
\sqrt{2}\,
\rho_{\varepsilon_j}*H.
\]
Use the common deterministic reference measures
\begin{equation}
\label{eq:common-critical-reference-measures}
\varrho_j(\mathrm{d}\theta)
:=
\sqrt{t_j}\,\mathrm{d}\theta
=
\sqrt{\log(1/\varepsilon_j)}\,\mathrm{d}\theta.
\end{equation}
The scale chaos measures are
\begin{equation}
\label{eq:scale-chaos-for-uniqueness}
\widetilde{\mu}_j(\mathrm{d}\theta)
:=
\exp\left(
U_j(\theta)
-
\frac{1}{2}
\mathbb{E}[U_j(\theta)^2]
\right)
\varrho_j(\mathrm{d}\theta)
=
\mu_{t_j}^H
(\mathrm{d}\theta).
\end{equation}
Proposition~\ref{prop:corrected-scale-critical-limit} yields \(\widetilde{\mu}_j\to\nu_H\) in probability,  and \(\nu_H\) is almost surely atomless.
The mollifier measures are
\begin{equation}
\label{eq:mollifier-chaos-for-uniqueness}
\begin{aligned}
\mu_j(\mathrm{d}\theta)
&:=
\exp\left(
V_j(\theta)
-
\frac{1}{2}
\mathbb{E}[V_j(\theta)^2]
\right)
\varrho_j(\mathrm{d}\theta)
\\
&=
\sqrt{\log(1/\varepsilon_j)}
\exp\left(
\sqrt{2}
\left(
\rho_{\varepsilon_j}*H
\right)(\theta)
-
\mathbb{E}\left[
\left(
\rho_{\varepsilon_j}*H
\right)(\theta)^2
\right]
\right)
\mathrm{d}\theta.
\end{aligned}
\end{equation}

The covariance kernels of \(U_j\) and \(V_j\) are \(2C_{t_j}^A\) and \(2C_{t_j}^B\), respectively.  
Therefore, Lemma~\ref{lem:scale-mollifier-covariance-comparison} gives
\[
\sup_j
\sup_{\theta,\theta'\in\mathbb{T}}
\left|
\operatorname{Cov}
\left(
U_j(\theta),U_j(\theta')
\right)
-
\operatorname{Cov}
\left(
V_j(\theta),V_j(\theta')
\right)
\right|
<
\infty,
\]
and, for every \(\delta>0\),
\[
\sup_{\substack{\theta,\theta'\in\mathbb{T}\\
d_{\mathbb{T}}(\theta,\theta')\geq\delta}}
\left|
\operatorname{Cov}
\left(
U_j(\theta),U_j(\theta')
\right)
-
\operatorname{Cov}
\left(
V_j(\theta),V_j(\theta')
\right)
\right|
\longrightarrow0.
\]

The circle is a compact doubling metric space, both field sequences are spatially smooth, the reference measures in \eqref{eq:common-critical-reference-measures} are identical,
and the scale limit is atomless.  
Thus, \cite[Theorem~1.1]{JunnilaSaksman2017} applies and yields
\begin{equation}
\label{eq:mollifier-chaos-distributional-convergence}
\mu_j
\longrightarrow
\nu_H
\qquad
\text{in distribution}.
\end{equation}

The probability-level identification requires the linear-regularization theorem.

\begin{lemma}
\label{lem:periodic-convolution-linear-regularization}
For the sequence \((U_j)_{j\geq1}\), define
\[
R_jf
=
\rho_{\varepsilon_j}*f.
\]
Then, \((R_j)_{j\ge 1}\) is a linear regularization process in the sense of \cite[Definition~4.3]{JunnilaSaksman2017}. 
Moreover,
\[
R_j\left(
\sqrt{2}H
\right)
=
V_j.
\]
\end{lemma}

\begin{proof}
For every Hölder-continuous function \(f\) on \(\mathbb{T}\),
\[
\|R_jf-f\|_{L^\infty(\mathbb{T})}
\longrightarrow0.
\]
This is the approximate-identity property.

Fix \(j\) and \(s>0\).  
By Lemma~\ref{lem:scale-field-negative-Sobolev-limit},
\[
H_{t_m}
\longrightarrow
H
\]
almost surely in \(H^{-s}(\mathbb{T})\). 
Convolution by the fixed smooth kernel \(\rho_{\varepsilon_j}\) maps \(H^{-s}(\mathbb{T})\) continuously into \(C^\infty(\mathbb{T})\).
Therefore,
\[
R_jU_m
\longrightarrow
\sqrt{2}\,
\rho_{\varepsilon_j}*H
=
V_j
\]
almost surely in \(C^\infty(\mathbb{T})\). 
This is the second condition in the definition of a linear regularization process.
\end{proof}

The scale increments are
\[
U_{j+1}-U_j
=
\sqrt{2}
\left(
X_{t_{j+1}}-X_{t_j}
\right),
\]
and are mutually independent.  
The remaining hypotheses of \cite[Theorem~4.4]{JunnilaSaksman2017} follow from the convergence \(\widetilde{\mu}_j\to\nu_H\) in probability,
Lemma~\ref{lem:periodic-convolution-linear-regularization}, and \eqref{eq:mollifier-chaos-distributional-convergence}.
Therefore,
\begin{equation}
\label{eq:mollifier-chaos-probability-convergence}
\mu_j
\longrightarrow
\nu_H
\qquad
\text{in probability}.
\end{equation}

To pass from the sequential convergence above to the full mollifier family,  for \(0<\varepsilon<1\), define
\[
\mu_{\varepsilon}^{H,\rho}(\mathrm{d}\theta)
=
\sqrt{\log(1/\varepsilon)}
\exp\left(
\sqrt{2}(\rho_\varepsilon*H)(\theta)
-
\mathbb{E}\left[
(\rho_\varepsilon*H)(\theta)^2
\right]
\right)
\mathrm{d}\theta.
\]
Thus, for the sequence fixed above, \( \mu_j=\mu_{\varepsilon_j}^{H,\rho}. \)

\begin{proposition}[Canonical critical chaos of the corrected field]
\label{prop:canonical-corrected-chaos}
For every nonnegative \(\rho\in C_c^\infty(\mathbb{R})\) of integral one,
\begin{equation}
\label{eq:full-corrected-mollifier-convergence}
\sqrt{\log(1/\varepsilon)}
\exp\left(
\sqrt{2}
\left(
\rho_\varepsilon*H
\right)(\theta)
-
\mathbb{E}\left[
\left(
\rho_\varepsilon*H
\right)(\theta)^2
\right]
\right)
\mathrm{d}\theta
\longrightarrow
e^{\sqrt{2}Y-\sigma_Y^2}
\sqrt{\frac{2}{\pi}}\,
M_*'
\end{equation}
in probability in the weak topology as \(\varepsilon\downarrow0\).
The limit is independent of the mollifier.
Denoting this common limit by \(M_H^{\mathrm{crit}}\), one has
\begin{equation}
\label{eq:canonical-corrected-chaos}
M_H^{\mathrm{crit}}
=
e^{\sqrt{2}Y-\sigma_Y^2}
\sqrt{\frac{2}{\pi}}\,
M_*'
\end{equation}
on the original coupling.
\end{proposition}

\begin{proof}
The argument above applies to every strictly decreasing deterministic sequence \(\varepsilon_j\downarrow0\)
and gives convergence in probability along that sequence to the same random measure \(\nu_H\).

If the full convergence failed, there would exist a metric \(d_{\mathrm{w}}\) generating the weak topology, constants \(\eta,c>0\),
and a sequence \(\varepsilon_j\downarrow0\) such that
\[
\mathbb{P}\left(
d_{\mathrm{w}}
\left(
\mu_{\varepsilon_j}^{H,\rho},
\nu_H
\right)
>\eta
\right)
\geq c
\]
for every \(j\). 
This contradicts \eqref{eq:mollifier-chaos-probability-convergence}.  
Thus, \eqref{eq:full-corrected-mollifier-convergence} holds.

The comparison applies to every admissible mollifier and always produces the same random measure \(\nu_H\).
Hence, this common mollifier-independent limit is \(M_H^{\mathrm{crit}}=\nu_H\).
Combining this identity with \eqref{eq:corrected-scale-limit-measure} gives \eqref{eq:canonical-corrected-chaos}.
\end{proof}

By Theorem~\ref{thm:periodized-derivative-rajchman} and Lemma~\ref{lem:rajchman-multiplier},
the canonical critical measure \(M_H^{\mathrm{crit}}\) is almost surely Rajchman.

\subsection{Proof of the main theorem}
\label{subsec:proof-main-theorem}

\begin{proof}[Proof of Theorem~\ref{thm:main-circle-rajchman}]
By Proposition~\ref{prop:corrected-field-law}, for every fixed mollifier and every \(\varepsilon>0\),
\[
\rho_\varepsilon*H
\stackrel{\mathrm{law}}{=}
\phi_\varepsilon+G_0.
\]
The corresponding finite-scale critical measures therefore have the same law.
Passing to their probability limits gives
\begin{equation}
\label{eq:critical-chaos-equality-in-law}
M_H^{\mathrm{crit}}
\stackrel{\mathrm{law}}{=}
M_{\phi+G_0}^{\mathrm{crit}}.
\end{equation}

The set of Rajchman measures is Borel in the weak topology on finite positive measures.
Indeed, it is
\[
\bigcap_{q=1}^{\infty}
\bigcup_{N=1}^{\infty}
\bigcap_{\substack{n\in\mathbb{Z}\\|n|\geq N}}
\left\{
\mu:
\left|
\int_{\mathbb{T}}
e^{in\theta}\,\mu(\mathrm{d}\theta)
\right|
<
\frac{1}{q}
\right\},
\]
and every fixed Fourier-coordinate map is weakly continuous.
By \eqref{eq:canonical-corrected-chaos}, Theorem~\ref{thm:periodized-derivative-rajchman} and Lemma~\ref{lem:rajchman-multiplier}, 
\(M_H^{\mathrm{crit}}\) is almost surely Rajchman.
Hence, \eqref{eq:critical-chaos-equality-in-law} implies that \(M_{\phi+G_0}^{\mathrm{crit}}\) is almost surely Rajchman.

Since \(G_0\) is spatially constant, independent of \(\phi\), and has variance \(a_0\), for every \(\varepsilon>0\),
\[
\begin{aligned}
&\sqrt{\log(1/\varepsilon)}
\exp\left(
\sqrt{2}\bigl(\phi_\varepsilon(\theta)+G_0\bigr)
-
\mathbb{E}\bigl[
\bigl(\phi_\varepsilon(\theta)+G_0\bigr)^2
\bigr]
\right)\mathrm{d}\theta
\\
&\quad=\ 
e^{\sqrt{2}G_0-a_0}
\sqrt{\log(1/\varepsilon)}
\exp\left(
\sqrt{2}\phi_\varepsilon(\theta)
-
\mathbb{E}\bigl[\phi_\varepsilon(\theta)^2\bigr]
\right)\mathrm{d}\theta.
\end{aligned}
\]
Passing to the probability limit and using Definition~\ref{def:critical-circle-chaos} gives
\begin{equation}
\label{eq:constant-mode-critical-chaos}
M_{\phi+G_0}^{\mathrm{crit}}
=
e^{\sqrt{2}G_0-a_0}
M_\phi^{\mathrm{crit}}
\end{equation}
on the original coupling.
The multiplying scalar is almost surely finite and strictly positive. 
Hence,  \(M_\phi^{\mathrm{crit}}\)  is almost surely Rajchman, and therefore
\[
\widehat{M_\phi^{\mathrm{crit}}}(n)
\longrightarrow0
\qquad
\text{almost surely as }|n|\to\infty.
\]
This proves Theorem~\ref{thm:main-circle-rajchman}.
\end{proof}

\subsection*{Funding}

X. F. is partially supported by the National Science and Technology Council, Taiwan, grant no.~114-2115-M-A49-003-MY3.

\end{document}